\documentclass[11pt]{article}
\usepackage{amsfonts}
\usepackage[latin9]{inputenc}
\usepackage{amsmath}
\usepackage{amssymb}
\usepackage{breqn}
\usepackage{url}
\usepackage{graphicx}
\usepackage[pdfstartview=FitH]{hyperref}
\usepackage{amsthm}
\usepackage{hyperref}
\usepackage{url}
\usepackage{enumitem}
\usepackage{fullpage}
\usepackage{verbatim}
\usepackage{bbm}

\newcommand{\im}{\operatorname{Im}}
\newcommand{\Spec}{\operatorname{Spec}}
\newcommand{\Spanv}{\operatorname{Span}}
\newcommand{\supp}{\operatorname{supp}}

\newcommand{\Span}{\operatorname{span}}
\newcommand{\F}{\operatorname{F}}

\makeatletter

\begin{document}
\newtheorem{theorem}{Theorem}[section]
\newtheorem{lemma}[theorem]{Lemma}
\newtheorem{definition}[theorem]{Definition}
\newtheorem{claim}[theorem]{Claim}
\newtheorem{example}[theorem]{Example}
\newtheorem{remark}[theorem]{Remark}
\newtheorem{proposition}[theorem]{Proposition}
\newtheorem{corollary}[theorem]{Corollary}
\newtheorem{observation}[theorem]{Observation}
\newcommand{\subscript}[2]{$#1 _ #2$}
\newtheorem*{theorem*}{Theorem}
\author{
Tali Kaufman
\footnote{Department of Computer Science, Bar-Ilan University, kaufmant@mit.edu, research supported by ERC and BSF.}
\and
Izhar Oppenheim
\footnote{Department of Mathematics, Ben-Gurion University of the Negev, Be'er Sheva 84105, Israel, izharo@bgu.ac.il, research supported by ISF.}
}
%
\title{High Order Random Walks: Beyond Spectral Gap}
\maketitle
%
%
%
%

\begin{abstract}
We study high order random walks in high dimensional expanders; namely, in complexes which are local spectral expanders. Recent works have studied the spectrum of high order walks and deduced fast mixing. However, the spectral gap of high order walks is inherently small, due to natural obstructions (called coboundaries) that do \underline{not} happen for walks on expander graphs.

In this work we go beyond spectral gap, and relate the shrinkage of a $k$-cochain by the walk operator, to its structure under the assumption of local spectral expansion. A simplicial complex is called an {\em one-sided local spectral expander}, if its links have large spectral gaps and a {\em two-sided local spectral expander} if its links have large two-sided spectral gaps. 

We show two Decomposition Theorems (one per one-sided/two-sided local spectral assumption) : For every $k$-cochain $\phi$ defined on an $n$-dimensional local spectral expander, there exists a decomposition of $\phi$ into ``orthogonal'' parts that are, roughly speaking, the ``projections'' on the $j$-dimensional cochains for $0 \leq j \leq k$. The random walk shrinks each of these parts by a factor of $\frac{k+1-j}{k+2}$ plus an error term that depends on the spectral expansion. 

When assuming one-sided local spectral gap, our Decomposition Theorem yields an {\em optimal} mixing for the high order random walk operator. Namely, negative eigenvalues of the
links do not matter! This improves over \cite{DK} that assumed two-sided spectral gap in the links to get optimal mixing. This improvement is {\em crucial} in a recent breakthrough \cite{ALGV} proving a conjecture of Mihail and Vazirani. Additionally, we get an optimal mixing for high order random walks on Ramanujan complexes (whose links are one-sided local spectral expanders).

When assuming two-sided local spectral gap, our Decomposition Theorem allows us to describe the whole spectrum of the random walk operator (up to an error term that is determined by the spectral gap) and give an explicit orthogonal decomposition of the spaces of cochains that approximates the decomposition to eigenspaces of the random walk operator.
\end{abstract}

%
%
%
%

\section{Introduction}


In this paper we study high order random walks on simplicial complexes whose links are spectral expanders. High order random walks are strongly related to PCP agreement tests; direct product testing and direct sum testing \cite{DK}. This relation influenced, in part, the study of high order random walks.

The focus of previous works~\cite{KM, DK} was on bounding the second largest eigenvalue (in absolute value) of the high order walk operator in complexes whose links are good spectral expanders. Namely, previous works have shown that in complexes with links that are good spectral expanders, every $k$-cochain that is orthogonal to the constant functions is shrinked by the $k$-order random walk operator $M_k^{+}$, i.e., by the walk that walks from a $k$-face to a $k$-face through a $(k+1)$ face. The shrinkage rate is immediately determined by the second largest eigenvalue of the walk operator. However, due to natural obstructions (such as coboundaries) this second largest eigenvalue cannot be very small.

However, It could well be the case that $k$-cochains with some specific structures are shrinked much better than the bound obtained by the bound on the second largest eigenvalue. This is similar in spirit to the small set expansion question in, say, the noisy hypercube~\cite{BarakGHMRS15}. Indeed, the noisy hypercube is not a good expander so we can not say that general sets expand well. However, methods beyond the spectral gap enabled showing that small sets of the noisy hypercube expand very well.

The focus of this work is to relate the {\em structure} of a $k$-cochain $\phi$ to the amount of its shrinkage by the random walk operator $M_k^{+}$, in complexes that are local spectral expanders. We provide two decomposition theorems that relates the amount of shrinkage of a $k$-cochain to its ``projections'' on the spaces of cochains of lower dimensions. Exact formulations of this statement are given below.

\subsection{On simplicial complexes and localization} A pure $n$-dimensional simplicial complex $X$ is a simplicial complex in which every simplex is contained in an $n$-dimensional simplex. In other words, it is an $(n+1)$-hypergraph with a closure property: for every hyperedge in the hypergraph, all of its subsets are also hyperedges in the hypergraph. The sets with $(i+1)$-elements are denoted $X(i)$, $0 \leq i \leq n$. The {\em one-skeleton} of the complex $X$ is its underlying graph obtained by $X(0) \cup X(1)$. A set $\tau \in X(i)$ is called a {\em face}. The {\em link} of $\tau$ denoted $X_{\tau}$ is the complex obtained by taking all faces in $X$ that contain $\tau$ and removing $\tau$ from them. Thus, if $\tau$ is of dimension $i$ (i.e. $\tau \in X(i)$ ) then $X_{\tau}$ is of dimension $n-i-1$.

For $-1 \leq k \leq n-1$, we denote by $C^k (X, \mathbb{R})$ the set of all functions $\phi : X (k) \rightarrow \mathbb{R}$. Abusing the terminology, we will call the space $C^k (X, \mathbb{R})$ the space of (non-oriented) $k$-cochains. Below we define an inner-product on $C^k (X,\mathbb{R})$ and denote by $\Vert . \Vert$ the norm induced by this inner-product (see exact details in Section \ref{Weighted simplicial complexes section}).

For every $-1 \leq i \leq n-2$, the one skeleton of $X_{\tau}$ is a graph. The second largest eigenvalue of the random walk on $X_\tau$ (with suitable weights - see below) is denoted $\mu_{\tau}$; the smallest eigenvalue of the random walk on $X_\tau$ (with suitable weights) is denoted $\nu_{\tau}$.

\begin{definition}[One sided local spectral expander]
A pure $n$-dimensional complex $X$ is called a {\em one-sided $\lambda$-local-spectral expander} if for every $-1 \leq i \leq n-2$, and for every $\tau \in X(i)$,  $\mu_{\tau} \leq \lambda$.
\end{definition}

\begin{definition}[Two sided local spectral expander]
A pure $n$-dimensional complex $X$ is a {\em two-sided $\lambda$-local-spectral expander} if for every $-1 \leq i \leq n-2$, and for every $\tau \in X(i)$, $-\lambda \leq \nu_{\tau}, \mu_{\tau} \leq \lambda$.
\end{definition}

\subsection{On high order random walks and related results} We study the random walk operator, $M_k^{+}$, which is the transition matrix of a random walk on the $k$-faces of $X$ defined as follows: given a $k$-dimensional face of $X$, $\tau \in X(k)$, choose randomly (according to a weight function) a $(k+1)$-face $\sigma \in X(k+1)$ such that $\tau \subset \sigma$ and then choose uniformly a $k$-face $\tau ' \in X(k)$ such that $\tau ' \subset \sigma$ (for an exact definition see Section \ref{HD RW and the signless differential sec}). We normalize the operator so that the largest eigenvalue is $1$.
Such walks were first introduced and studied in $\cite{KM}$ where it was shown that that in one-sided $\lambda$-local-spectral expanders a $k$-cochain $\phi$ orthogonal to the constant functions satisfies the following:

$$\Vert M_k^{+} \phi \Vert \leq \left( 1 - \frac{\varepsilon (k,\lambda)}{2(k+2)^2} \right)||\phi||,$$
where $\varepsilon (k,\lambda) \leq 1$.


In the recent work of \cite{DK} it was shown that for two-sided $\lambda$-local-spectral expanders, a $k$-cochain $\phi$ orthogonal to the constant functions satisfies the following:
$$\Vert M_k^{+} \phi \Vert \leq \left( \left( 1 - \frac{1}{k+2} \right) +  O \left((k+1)\lambda \right) \right)||\phi|| .$$

This bound can be shown to be optimal in the sense that it is essentially equivalent (for small enough $\lambda$) to the bound obtained in the complete complex. The work of \cite{DK} have used this optimal bound in order to obtain a complete de-randomization of the direct product testing. The method that they introduced is the method of decreasing differences. The seemingly mild improvement that \cite{DK} achieves over \cite{KM} is crucial for their application. Note that~\cite{DK} requires two-sided $\lambda$-local-spectral expanders. As the Ramanujan complexes are only one-sided $\lambda$-local-spectral expanders, the result of \cite{DK} does not apply to the Ramanujan complexes themselves but only to other complexes that could be built from them.

\subsection{Decomposition theorems for high order random walks and their implications}
In this paper we show two decomposition theorems: the first for one-sided $\lambda$-local-spectral expanders and the second for two-sided $\lambda$-local-spectral expanders. Both theorems follow the same philosophy: a $k$-cochain $\phi$ orthogonal to the constants can be decomposed into ``orthogonal'' parts $\phi^0$,...,$\phi^k$, where, roughly speaking, each such part is the ``projection'' on the $j$-dimensional cochains for $0 \leq j \leq k$ and applying $M_k^+$ on the each such $\phi^j$ shrinks it by a factor of $\frac{k+1-j}{k+2} + O (\lambda)$. The different assumptions on the local spectral gaps (one-sided or two-sided) lead to different implementations of this philosophy.

\begin{theorem} [Decomposition Theorem for one-sided local spectral expanders, informal, for formal see Theorem \ref{Decomp thm}, Corollary \ref{bound on d phi using spectral gaps}]

Given a pure $n$-dimensional one-sided $\lambda$-local-spectral expander $X$. For a $k$-cochain $\phi$ ($k \leq n-1$) orthogonal to the constant functions there exist $j$-cochains $\phi^j$ for every $0 \leq j \leq k$ such that $||\phi||^2 = \sum_{j=0}^{k} ||\phi^j||^2$ and
$$\langle M_k^{+} \phi, \phi \rangle \leq \sum_{j=0}^{k} \left( \frac{k+1-j}{k+2} + f(k,j) \lambda \right) ||\phi^j||^2,$$
where $f(k,j)$ are explicit positive constants (independent of $X$).
\end{theorem}

As a corollary of the decomposition theorem we derive optimal bounds on the second largest eigenvalue (in absolute value) of $M_k^{+}$ for $X$ which is a one-sided $\lambda$-local-spectral expander. This result is stronger than~\cite{DK} that applies only for \underline{two}-sided $\lambda$-local-spectral expander.

\begin{theorem} [Bounding the second eigenvalue of the $k$-walk Theorem, informal, for formal see Theorem \ref{mixing of random walks thm}] \label{mixing of random walks informal thm}

Given a pure, $n$-dimensional, one-sided $\lambda$-local-spectral expander $X$. For a $k$-cochain $\phi$ ($k \leq n-1$) orthogonal to the constant functions:

$$\Vert M_k^{+} \phi \Vert \leq \left( \left(1 - \frac{1}{k+2} \right) +  \frac{k+1}{2} \lambda \right) ||\phi|| .$$

\end{theorem}

\paragraph{On the usefullness of Theorem \ref{mixing of random walks informal thm} in the recent breakthrough \cite{ALGV}.} 
Recently, Theorem \ref{mixing of random walks informal thm} was used by Anari, Liu, Gharan and Vinzant \cite{ALGV} in their proof of a famous conjecture of Mihail and Vazirani. We note that in the cases considered in their proof, local two-sided spectral expansion does not hold and therefore it was important that the bound on mixing rate of the $k$-random walk can be derived only from the local one-sided spectral expansion.

\bigskip

The Decomposition Theorem in the case of one-sided local spectral expanders has two problems: first, the ``orthogonal'' parts here are at different dimensions, i.e., we do not really have an orthogonal decomposition of the spaces of $k$-cochains into orthogonal subspaces. Second, while the $j$-cochains $\phi^j$ are described explicitly, their definition involves taking a pre-image of the square root of $M_k^{+}$ and thus their actual computation is hard to implement. These two problems are resolved when assuming two-sided local spectral gap. In that cases, under the assumption that the local spectral gap is sufficiently small, we show that the spectrum of $M_k^+$ is contained in small neighborhoods around $\frac{k+1-j}{k+2}$ for $j=0,...,k$ and that the eigenspaces for these eigenvalues can be approximated by an explicit orthogonal decomposition of the space of $k$-cochains.

In order to state this result precisely, we will need some additional definitions and notations: denote by $C^k (X)$ the spaces of $k$-cochains, i.e., maps of the form $\phi : X(k) \rightarrow \rightarrow \mathbb{R}$ and denote by $C_0^k (X)$ be the subspace of $k$-cochains orthogonal to the constant functions (the exact inner-product is defined in Section \ref{Weighted simplicial complexes section} below). For $0 \leq j < k \leq n$, let $d_{j \nearrow k} : C^j (X) \rightarrow C^k (X)$ to be the operator induced by the incidence matrix of $j$-simplices in $k$-simplices, i.e., for every $\sigma \in C^k (X)$ and every $\phi \in C^j (X)$, $d_{j \nearrow k} \phi (\sigma) = \sum_{\tau \in X(j), \tau \subset \sigma} \phi (\tau)$.
Define
$$U_k^j = \begin{cases}
d_{0 \nearrow k} (C_0^0 (X)) & j=0 \\
d_{j \nearrow k} (C_0^j (X)) \cap  (d_{j-1 \nearrow k} (C_0^{j-1} (X)))^\perp & j=1,...,k-1 \\
(d_{k-1 \nearrow k} (C_0^{k-1} (X)))^\perp & j =k
\end{cases}.$$
We show that for every $0 \leq k \leq n-1$, $U_k^0 \oplus ... \oplus U_k^k$ is an orthogonal decomposition of $C_0^k (X)$ (this is a general result that does not really on any spectral assumptions) and that if $X$ is a two-sided $\lambda$-local spectral expander with $\lambda$ small enough, the spaces $U_k^j$ approximate the eigenspaces of $M_k^+$:

\begin{theorem} [Decomposition Theorem for two-sided local spectral expanders, informal, for formal see Corollary \ref{almost e.s. corollary}, Theorem \ref{e.v. bounds thm}, Theorem \ref{approx of W_k by U_k thm}, Corollary \ref{rw shrink coro}] \label{almost e.s. thm-informal}
\label{two-sided decomp thm - intro}
Let $X$ is a two-sided $\lambda$-local spectral expander. If $\lambda$ is sufficiently small there are explicit constants $f(k), g(k), h(k)$ such that the following holds:
\begin{enumerate}
\item For every $0 \leq k \leq n-1$, every $0 \leq j \leq k$ and every $\phi \in C_0^k$, the projection of $\phi$ on $U_k^j$ is an almost $\frac{k+1-j}{k+2}$ eigenvector, i.e.,
$$\left\Vert M_k^+ (P_{U^j_k} \phi) - (\frac{k+1-j}{k+2}) (P_{U^j_k} \phi) \right\Vert \leq f(k) \lambda \Vert P_{U^j_k} \phi \Vert,$$
where $P_{U^j_k} \phi$ is the orthogonal projection on $U^j_k$.
\item For every $0 \leq k \leq n-1$, the non-trivial spectrum of $M_k^+$ is concentrated in small neighborhood of $\frac{k+1-j}{k+2}$, $j=0,...,k$, explicitly
$$\Spec (M_k^+) \subseteq \lbrace 1 \rbrace \cup \bigcup_{j=0}^{k} \left[\frac{k+1-j}{k+2} -   g(k) \lambda, \frac{k+1-j}{k+2} +   g(k) \lambda \right].$$
\item For every $0 \leq k \leq n$, the shrinkage of the random walk $M_k^+$ after any number of steps can be determined based on the size of the projections on the $U_k^j$'s: for every $i \in \mathbb{N}$ and every $\phi \in C_0^k (X)$,
$$\Vert (M_k^+)^i \phi \Vert \leq \sqrt{\sum_{j=0}^k \left( \left(  \frac{k+1-j}{k+2} + \frac{\sqrt{k+1}}{k+2}\varepsilon_k (\lambda) \right)^{2i} +  b_k (i, \lambda) \right) \left\Vert P_{U_k^j} \phi \right\Vert^2},$$
and
$$\Vert (M_k^+)^i \phi \Vert \geq \sqrt{\sum_{j=0}^k \left(  \frac{k+1-j}{k+2} - \frac{\sqrt{k+1}}{k+2}\varepsilon_k (\lambda) \right)^{2i} (1 - b_k (i, \lambda)) \left\Vert P_{U_k^j} \phi \right\Vert^2},$$
where $\varepsilon_k (\lambda), b_k (i, \lambda)$ are constants that tend to $0$ as $\lambda$ tends to $0$.
\end{enumerate}

\end{theorem}

\paragraph{On the relation of Theorem \ref{almost e.s. thm-informal} to the work of \cite{DDFH}. } In an earlier version of this article, we proved only the Decomposition Theorem for one-sided local spectral expanders. In a subsequent work, Dikstein, Dinur, Filmus and  Harsha \cite{DDFH} proved a decomposition theorem for two-sided local spectral expander that is similar to the first assertion Theorem \ref{two-sided decomp thm - intro}. Theorem \ref{two-sided decomp thm - intro} is based on the ideas of \cite{DDFH}, but, in our opinion, it gives a more complete treatment to the same problem. For comparison, in \cite{DDFH}, the subspaces in the decomposition are almost-orthogonal (and not orthogonal) and it does not include the results regarding the spectrum and the shrinkage of $M_k^+$ that are stated in the second and third assertions of Theorem \ref{two-sided decomp thm - intro}.

\bigskip

\subsection{On small set expansion phenomenon, the Grassmann complex and our work}

As we have explained above, we study the amount of shrinkage of a $k$-cochain by the random walk operator $M_k^{+}$. Our motivation is to go beyond spectral gap and to relate the shrinkage of a $k$-cochain by the operator, to its structure. Similar questions are asked in the study of small set expansion in the noisy hypercube \cite{BarakGHMRS15}. Recently it was shown that studying the structure of non expanding $k$-cochains of the Grassmann complex is strongly related to the "2-to-1 games Conjecture" \cite{DinurKKMS16,DinurKKMS17}, which is a weaker form of the famous Unique Game Conjecture. Our work here, is of the same flavor. However, instead of working with a specific complex (e.g., the Grassmann complex) we work with simplicial complexes, whose links are good spectral expanders. We characterize non expanding $k$-cochains as those whose mass is concentrated  on the lower levels of the decomposition that we construct.

\subsection{On different definitions of local spectral expanders}
In \cite{Opp-LocalSpectral}, the second author gave a different (not strictly equivalent) definition for the notion of $\lambda$-local spectral expansion. The definition in \cite{Opp-LocalSpectral} goes as follows: a simplicial complex $X$ is called a one sided $\lambda$-local spectral expansion in  \cite{Opp-LocalSpectral} if all its links of dimension $>0$ are connected and if for every $\tau \in X(n-2)$, $\mu_{\tau} \leq \lambda$. Also, a simplicial complex $X$ is called a two sided $(\lambda, \kappa)$-local spectral expansion in  \cite{Opp-LocalSpectral} if all its links of dimension $>0$ are connected and if for every $\tau \in X(n-2)$, $\mu_{\tau} \leq \lambda$ and $\nu_\tau \geq \kappa$.

Although these definitions are not strictly equivalent, a main result in \cite{Opp-LocalSpectral} (see also Corollary \ref{Explicit spec descent coro} below) shows that they are equivalent up to changing $\lambda$, i.e., for any $0 <\lambda$ there is $0<\lambda' = \frac{\lambda}{1+(n-1)\lambda}$ such that for every pure $n$-dimensional simplicial complex $X$
\begin{itemize}
\item $X$ is a one-sided $\lambda$ local spectral expander by the definition of this paper if and only if $X$ is a one-sided $\lambda'$ local spectral expander by the definition of \cite{Opp-LocalSpectral}.
\item $X$ is a two-sided $\lambda$ local spectral expander by the definition of this paper if and only if $X$ is a two-sided $(\lambda', - \lambda')$ local spectral expander by the definition of \cite{Opp-LocalSpectral}.
\end{itemize}
This equivalence up to changing $\lambda$ is useful, because in examples it is sometimes easier to bound just the spectrum of the $1$-dimensional links and not to have to analyse the spectrum of all the links.

\subsection{Organisation} The paper is organized as follows. Section \ref{Weighted simplicial complexes section} contains the basic definitions regarding weighted complexes and the inner product and norm induced by the weights. Section \ref{HD RW and the signless differential sec} contains the definitions of the upper and lower random walks, and what we call the ``signless differential'' and show how these definitions are connected. Section \ref{Links and localization sec} contains a connection in the spirit of the so called ``Garland method'' between the norm of the signless differential and the norm of the upper random walk operator in the links. Section \ref{Decomposition theorem for upper random walks sec} contains our main results regarding decomposition theorems for the upper random walk and their corollaries.



%
%
%




%
%
%
%
%
%
%
%
%
%

\section{Weighted simplicial complexes}
\label{Weighted simplicial complexes section}
Let $X$ be a pure $n$-dimensional finite simplicial complex. For $-1 \leq k \leq n$, define $X (k)$ to be the set of all $k$-simplices in $X$ ($X (-1) = \lbrace \emptyset \rbrace$). A \textit{weight function $m$} on $X$ is a function:
$$m : \bigcup_{-1 \leq k \leq n} X (k) \rightarrow \mathbb{R}^+,$$
such that for every $-1 \leq k \leq n-1$ and for every $\tau \in X (k)$ we have that
$$m (\tau) = \sum_{\sigma \in X (k+1), \tau \subseteq \tau} m(\sigma).$$
By its definition, it is clear that $m$ is determined by the values it takes on $X (n)$. A simplicial complex with a weight function will be called a weighted simplcial complex. The most basic (and most important) example of a weight function is the \textit{homogeneous weight} $m$ that is the weight function defined by $m(\sigma) =1$ for every $\sigma \in X(n)$. The following facts already appear in \cite{Opp-LocalSpectral} and therefore the proofs are omitted.

\begin{proposition}{\cite[Proposition 2.7]{Opp-LocalSpectral}}
\label{weight in n dim simplices}
For every $-1 \leq k \leq n$ and every $\tau \in X (k) $ we have that
$$\dfrac{1}{(n-k)!}  m (\tau) =\sum_{\sigma \in X (n), \tau \subseteq \sigma} m (\sigma ),$$
where $\tau \subseteq \sigma$ means that $\tau$ is a face of $\sigma$.

In particular, the homogeneous weight $m$ on $X$ can be written explicitly as
$$ \forall -1 \leq k \leq n, \forall \tau \in X (k), \dfrac{1}{(n-k)!}  m(\tau) =  \vert \lbrace \sigma \in X (n) : \tau \subseteq \sigma \rbrace \vert .$$

\end{proposition}

\begin{corollary}{\cite[Corollary 2.8]{Opp-LocalSpectral}}
\label{weight in l dim simplices}
For every $-1 \leq k < l \leq n$ and every $\tau \in X (k) $ we have
$$\dfrac{1}{(l-k)!} m(\tau) = \sum_{\sigma \in X (l), \tau \subset \sigma} m(\sigma) .$$
\end{corollary}

Throughout this article, $X$ is a pure $n$-dimensional finite weighted simplicial complex with a weight function $m$.

Recall that in the introduction we defined the space of $k$-cochains, $C^k (X, \mathbb{R})$, as the set of all functions $\phi : X (k) \rightarrow \mathbb{R}$. On $C^k (X,\mathbb{R})$ define the following inner-product:
$$\forall \phi, \psi \in C^k (X,\mathbb{R}), \langle \phi, \psi \rangle = \sum_{\sigma \in  X (k)} m (\sigma ) \phi (\sigma) \psi (\sigma).$$
Denote by $\Vert . \Vert$ the norm induced from this inner-product.

\section{Higher dimensional random walks and the signless differential}
\label{HD RW and the signless differential sec}
\subsection{Upper and lower random walks}

For $X$ as above, we define the following random walks on simplices of $X$:

\begin{definition}
For $0 \leq k \leq n-1$, the upper random walk on $k$-simplices is defined by the transition probability matrix $M^+_k : X (k) \times X (k) \rightarrow \mathbb{R}$:
$$M^+_k (\tau, \tau') = \begin{cases}
\frac{1}{k+2} & \tau = \tau' \\
\frac{m (\tau \cup \tau')}{(k+2) m(\tau)} & \tau \cup \tau' \in X (k+1) \\
0 & \text{otherwise}
\end{cases}.$$
\end{definition}

\begin{definition}
For $0 \leq k \leq n$, the lower random walk on $k$-simplices is defined by the transition probability matrix $M^-_k : X (k) \times X (k) \rightarrow \mathbb{R}$:
$$M^-_k (\tau, \tau') = \begin{cases}
\sum_{\eta \in X (k-1)} \frac{m(\tau)}{(k+1) m (\eta)} & \tau = \tau' \\
\frac{m (\tau')}{(k+1) m(\tau \cap \tau')} & \tau \cap \tau' \in X (k-1) \\
0 & \text{otherwise}
\end{cases}.$$
\end{definition}

We leave it to the reader to check that those are in fact transition probability matrix, i.e., that for every $\tau$, $\sum_{\tau'} M^\pm_k (\tau,\tau') =1$. We note that both random walks defined above are lazy in the sense that $M^\pm (\tau, \tau) \neq 0$. In the case of the upper random walk, one can easily define a non lazy random walk as follows:
\begin{definition}
\label{non lazy walk def}
For $0 \leq k \leq n-1$, the non-lazy upper random walk on $k$-simplices is defined by the transition probability matrix $(M')^+_{k} : X (k) \times X (k) \rightarrow \mathbb{R}$:
$$(M')_{k}^+ = \frac{k+2}{k+1} \left( M^+_k - \frac{1}{k+2} I \right) = \frac{k+2}{k+1} M^+_k - \frac{1}{k+1} I .$$
\end{definition}

It is standard to view $M^{\pm}_k, (M')^+_{k}$ as averaging operators on $C^k (X,\mathbb{R})$ and we will not make the distinction between the transition probability matrix and the averaging operator it induces.

It is worth noting that $M^{-}_0$ and $(M')^+_{0}$ are familiar operators/matrices: $M^-_0$ is  a projection on the space of the constant functions (on vertices) with respect to the inner-product defined above, and $(M')^+_{0}$ is the weighted (normalized) adjacency matrix of the $1$-skeleton of $X$.

\subsection{The signless differential}
We recall that when considering oriented simplicial complexes, the upper and lower Laplacians are obtained for the differential (i.e., coboundary) operator using discrete Hodge theory (see for instance \cite[Section 2.2]{Opp-LocalSpectral}). In our setting, disregarding orientation yields a similar result: below we define a signless differential and show that the upper and lower random walk operators are obtained from this signless differential in exactly the same way that the upper and lower Laplacians are obtained for the usual differential in the oriented setting - see exact formulation in Corollary \ref{connection between d*d and M coro} below.

\begin{definition}
For $-1 \leq k \leq n-1$, the signless $k$-differential is an operator $d_k : C^k (X,\mathbb{R}) \rightarrow C^{k+1} (X,\mathbb{R})$ defined as:
$$\forall \phi \in C^k (X,\mathbb{R}), \forall \sigma \in X (k+1), d_k \phi (\sigma) = \sum_{\tau \subset \sigma, \tau \in X (k)} \phi (\tau).$$
Define $(d_{k})^* : C^{k+1} (X,\mathbb{R}) \rightarrow  C^{k} (X,\mathbb{R})$ to be the adjoint operator to $d_k$, i.e., the operator such that for every $\phi \in  C^{k} (X,\mathbb{R}), \psi \in C^{k+1} (X,\mathbb{R})$, $\langle d_k \phi, \psi \rangle = \langle \phi, (d_{k})^* \psi \rangle$.
\end{definition}

\begin{remark}
We note that the signless differential is not a differential in the usual sense, since $d_{k+1} d_k \neq 0$. The name signless differential stems from the fact that this is the operator we will use in lieu of the differential in our setting (note that since our non-oriented cochains are defined without using orientation of simplices, we cannot use the usual differential).
\end{remark}

Below, we will usually omit the index of signless differential and its adjoint and just denote $d, d^*$ where $k$ will be implicit.

\begin{lemma}
\label{d* lemma}
For $-1 \leq k \leq n-1$, $d^* : C^{k+1} (X,\mathbb{R}) \rightarrow  C^{k} (X,\mathbb{R})$ is the operator
$$\forall \psi \in C^{k+1} (X,\mathbb{R}), \forall \tau \in X (k), d^* \psi (\tau) = \sum_{\sigma \in X (k+1), \tau \subset \sigma} \dfrac{m(\sigma)}{m(\tau)} \psi (\sigma).$$
\end{lemma}

\begin{proof}
The proof is very similar to the proof of \cite[Proposition 2.11]{Opp-LocalSpectral}, in which an analogues fact in proven in the oriented setting. We leave the adaption to the non-oriented setting to the reader.
\end{proof}

\begin{corollary}
\label{connection between d*d and M coro}
For $0 \leq k \leq n-1$ and $\phi \in C^k (X,\mathbb{R})$, $d^*d \phi = (k+2) M^{+} \phi$ and $d d^* \phi= (k+1) M^{-} \phi$.
\end{corollary}

\begin{proof}
Let $\phi \in C^k (X,\mathbb{R})$ and $\tau \in X (k)$, then
\begin{dmath*}
d^* d \phi (\tau) = \sum_{\sigma \in X (k+1), \tau \subset \sigma} \dfrac{m(\sigma)}{m(\tau)} d \phi (\sigma) =
\sum_{\sigma \in X (k+1), \tau \subset \sigma} \dfrac{m(\sigma)}{m(\tau)} \sum_{\tau' \in X (k), \tau' \subset \sigma} \phi (\tau') =
\sum_{\sigma \in X (k+1), \tau \subset \sigma} \dfrac{m(\sigma)}{m(\tau)} \sum_{\tau' \in X (k), \tau' \subset \sigma, \tau' \neq \tau} \phi (\tau') + \sum_{\sigma \in X (k+1), \tau \subset \sigma} \dfrac{m(\sigma)}{m(\tau)} \phi (\tau).
\end{dmath*}
Note that
$$\sum_{\sigma \in X (k+1), \tau \subset \sigma} \dfrac{m(\sigma)}{m(\tau)} \phi (\tau) = \dfrac{m(\tau)}{m(\tau)} \phi (\tau) = \phi (\tau).$$
Also note that
\begin{dmath*}
\sum_{\sigma \in X (k+1), \tau \subset \sigma} \dfrac{m(\sigma)}{m(\tau)} \sum_{\tau' \in X (k), \tau' \subset \sigma, \tau' \neq \tau} \phi (\tau')  =
\sum_{\sigma \in X (k+1), \tau \subset \sigma} \sum_{\tau' \in X (k), \tau' \subset \sigma, \tau' \neq \tau} \dfrac{m(\tau' \cup \tau)}{m(\tau)} \phi (\tau') =
\sum_{\tau' \in X (k), \tau \cup \tau' \in X (k+1)} \dfrac{m(\tau \cup \tau')}{m(\tau)} \phi (\tau').
\end{dmath*}
Therefore
$$d^* d \phi (\tau) = \phi (\tau) + \sum_{\tau' \in X (k), \tau \cup \tau' \in X (k+1)} \dfrac{m(\tau \cup \tau')}{m(\tau)} \phi (\tau') = (k+2) M^+ \phi (\tau).$$
Similarly,
\begin{dmath*}
d d^* \phi (\tau) = \sum_{\eta \in X (k-1), \eta \subset \tau} d^* \phi (\eta) =
\sum_{\eta \in X (k-1), \eta \subset \tau} \sum_{\tau' \in X (k), \eta \subset \tau'} \dfrac{m(\tau')}{m(\eta)} \phi (\tau') =
\sum_{\eta \in X (k-1), \eta \subset \tau} \sum_{\tau' \in X (k), \tau' \neq \tau, \eta \subset \tau'} \dfrac{m(\tau')}{m(\eta)} \phi (\tau') + \sum_{\eta \in X (k-1), \eta \subset \tau} \dfrac{m(\tau)}{m(\eta)} \phi (\tau) =
\sum_{\eta \in X (k-1), \eta \subset \tau} \sum_{\tau' \in X (k), \tau' \cap \tau = \eta} \dfrac{m(\tau')}{m(\tau \cap \tau')} \phi (\tau') + \sum_{\eta \in X (k-1), \eta \subset \tau} \dfrac{m(\tau)}{m(\eta)} \phi (\tau) =
\sum_{\tau' \in X (k), \tau \cap \tau' \in X (k-1)} \dfrac{m(\tau')}{m(\tau \cap \tau')} \phi (\tau') + \sum_{\eta \in X (k-1), \eta \subset \tau} \dfrac{m(\tau)}{m(\eta)} \phi (\tau) = (k+1) M^- \phi (\tau).
\end{dmath*}
\end{proof}

\section{Links and localization}
\label{Links and localization sec}

Let $X$ be a pure $n$-dimensional finite simplicial complex with a weight function $m$. Recall that for $-1 \leq k \leq n-1$, $\tau \in X (k)$, the link of $\tau$, denoted $X_\tau$, is a pure $(n-k-1)$-simplicial complex defined as:
$$\eta \in X_\tau (l) \Leftrightarrow \eta \in X (l) \text{ and } \tau \cup \eta \in X (k+l+1).$$

On $X_\tau$ we define the weight function $m_\tau$ induced by $m$ as
$$m_{\tau} (\eta) = m (\tau \cup \eta).$$
Using this weight function, the inner-product and the norm on $C^l (X_\tau, \mathbb{R})$ are defined as above. The operators $M^\pm_{\tau,l}, (M')^+_{\tau,l}$ and $d_\tau, d^*_\tau$ are also defined on $C^l (X_\tau, \mathbb{R})$ as above.

Given a cochain $\phi \in C^l (X, \mathbb{R})$ and a simplex $\tau \in X (k)$ with $-1 \leq k < l$, we define the localization of $\phi$ on $X_\tau$, denoted $\phi_\tau$ as a cochain $\phi_\tau \in C^{l-k-1} (X_\tau, \mathbb{R})$ defined as
$$\phi_\tau (\eta) = \phi (\tau\cup \eta).$$
We note that $\tau$ and $\eta$ are, by definition, disjoint and $\tau\cup \eta = \tau \dot\cup \eta$. However, to avoid cumbersome notation, we will not use the disjoint union symbol.

The key observation (which was initially due to Garland \cite{Garland}, but is now considered standard - see \cite{BS}, \cite{GW}) is that the inner-products of $\phi$, $d^* \phi$, and $d \phi$ can be calculated via their localizations:
\begin{proposition}
\label{localization proposition}
Let $-1 \leq k <l \leq n$ and let $\phi, \psi \in C^l (X,\mathbb{R})$, then
\begin{enumerate}
\item $${l+1 \choose k+1} \langle  \phi ,\psi \rangle= \sum_{\tau \in X (k)} \langle \phi_\tau, \psi_\tau \rangle$$
and in particular for $\psi = \psi$,
${l+1 \choose k+1} \Vert  \phi \Vert^2 = \sum_{\tau \in X (k)} \Vert \phi_\tau \Vert^2.$
\item $${l \choose k+1} \langle d^* \phi, d^* \psi  \rangle = \sum_{\tau \in X (k)} \langle d^*_\tau \phi_\tau , d^*_\tau \psi_\tau \rangle$$
and in particular for $\psi = \psi$,
${l \choose k+1} \Vert d^* \phi \Vert^2 = \sum_{\tau \in X (k)} \Vert d^*_\tau \phi_\tau \Vert^2.$
\item If $l<n$, then
$$\langle d \phi , d \psi \rangle = \sum_{\tau \in X (l-1)} \left( \langle d_\tau \phi_\tau ,d_\tau \psi_\tau \rangle  - \dfrac{l}{l+1} \langle \phi_\tau , \psi_\tau \rangle \right)$$
and in particular for $\psi = \psi$,
$\Vert d \phi \Vert^2 = \sum_{\tau \in X (l-1)} \left( \Vert d_\tau \phi_\tau \Vert^2 - \dfrac{l}{l+1} \Vert \phi_\tau \Vert^2 \right).$
\end{enumerate}

\end{proposition}

\begin{proof}
The facts stated in this Proposition were already proven in the oriented setting in \cite{Opp-LocalSpectral} (see {\cite[Lemma 3.4, Lemma 3.5]{Opp-LocalSpectral}}). The proofs in the non-oriented setting are very similar and are detailed in the appendix.
\end{proof}

As a result of Proposition \ref{localization proposition} we deduce the following:

\begin{proposition}
\label{d^2 proposition}
Let $0 \leq k  \leq n-1$ and let $\phi, \psi \in C^k (X,\mathbb{R})$, then
$$\langle d \phi,  d \psi \rangle = \langle d^* \phi, d^* \psi \rangle + \langle \phi, \psi \rangle + \sum_{\tau \in X (k-1)}    \langle (M ')^+_{\tau, 0} (I-M^-_{\tau, 0}) \phi_\tau , \psi_\tau \rangle$$
and in particular for $\phi = \psi$,
$$\Vert d \phi \Vert^2 = \Vert d^* \phi \Vert^2+ \Vert \phi \Vert^2 + \sum_{\tau \in X (k-1)}    \langle (M ')^+_{\tau, 0} (I-M^-_{\tau, 0}) \phi_\tau , \phi_\tau \rangle.$$
\end{proposition}

\begin{proof}
Let $\phi, \psi \in C^k (X,\mathbb{R})$. Note that for every $\tau \in X (k-1)$, $M^-_{\tau, 0}$ is the orthogonal projection on the space of constant functions in $C^0 (X_\tau, \mathbb{R})$ and therefore $(M ')^+_{\tau, 0} M^-_{\tau, 0} = M^-_{\tau, 0}$.

By Corollary \ref{connection between d*d and M coro}, for every $\tau \in X(k-1)$,
\begin{dmath*}
\langle d_\tau \phi_\tau ,d_\tau \psi_\tau  \rangle = \langle  2 M^+_{\tau, 0} \phi_\tau, \psi_\tau \rangle = \langle   ((M ')^+_{\tau, 0} +I) \phi_\tau, \psi_\tau \rangle =
\langle   (M ')^+_{\tau, 0} \phi_\tau, \psi_\tau \rangle + \langle \phi_\tau , \psi_\tau \rangle =
\langle   (M ')^+_{\tau, 0}  M^-_{\tau, 0} \phi_\tau, \psi_\tau \rangle + \langle   (M ')^+_{\tau, 0}  (I-M^-_{\tau, 0}) \phi_\tau, \psi_\tau \rangle + \langle \phi_\tau , \psi_\tau \rangle =
\langle   M^-_{\tau, 0} \phi_\tau , \psi_\tau \rangle + \langle   (M ')^+_{\tau, 0}  (I-M^-_{\tau, 0}) \phi_\tau, \psi_\tau \rangle + \langle \phi_\tau , \psi_\tau \rangle =
\langle  d^* \phi_\tau , d^* \psi_\tau \rangle + \langle   (M ')^+_{\tau, 0}  (I-M^-_{\tau, 0}) \phi_\tau, \psi_\tau \rangle + \langle \phi_\tau , \psi_\tau \rangle.
\end{dmath*}
Combining this equality with Proposition \ref{localization proposition} yields
\begin{dmath*}
\langle d \phi , d \psi \rangle = \sum_{\tau \in X (k-1)} \left( \langle d_\tau \phi_\tau ,d_\tau \psi_\tau \rangle  - \dfrac{k}{k+1} \langle \phi_\tau , \psi_\tau \rangle \right) =
\sum_{\tau \in X (k-1)} \left( \langle  d^* \phi_\tau , d^* \psi_\tau \rangle + \langle   (M ')^+_{\tau, 0}  (I-M^-_{\tau, 0}) \phi_\tau, \psi_\tau \rangle + \langle \phi_\tau , \psi_\tau \rangle  - \dfrac{k}{k+1} \langle \phi_\tau , \psi_\tau \rangle \right) =
\langle  d^* \phi , d^* \psi \rangle +  \langle  \phi , \psi\rangle + \sum_{\tau \in X (k-1)} \left(  \langle   (M ')^+_{\tau, 0}  (I-M^-_{\tau, 0}) \phi_\tau, \psi_\tau \rangle  \right),
\end{dmath*}
as needed.
\end{proof}

In light of the above corollary, we will want to bound the expression
$$\sum_{\tau \in X (k-1)}    \langle (M ')^+_{\tau, 0} (I-M^-_{\tau, 0}) \phi_\tau , \psi_\tau \rangle,$$
using spectral information about $X$. To make this precise, we will recall/define the following. For $0 \leq k \leq n-1$ and $\tau \in X (k-1)$, recall that by corollary \ref{connection between d*d and M coro}, $(M ')^+_{\tau, 0} = d_\tau^* d_\tau - I$ and therefore the eigenvalues of $(M ')^+_{\tau, 0}$ are real. Denote $\mu_\tau$ to be the second largest eigenvalue of $(M ')^+_{\tau, 0}$ and $\nu_\tau$ to be the smallest eigenvalue of $(M ')^+_{\tau, 0}$. Note that if $1$-skeleton of $X_\tau$ is connected, then for every eigenfunction $\varphi \in C^0 (X_\tau, \mathbb{R})$, if $\varphi \perp \im M^-_{\tau, 0}$, then $(M ')^+_{\tau, 0} \varphi = \mu \varphi$ with $\nu_\tau \leq \mu \leq \mu_\tau <1$. Denote
$$\mu_k = \max_{\tau \in X (k-1)} \mu_\tau, \nu_k = \min_{\tau \in X (k-1)} \nu_\tau.$$

\begin{lemma}
\label{bounding the sum of innerproducts lemma}
For every $0 \leq k \leq n-1$ and every $\phi, \psi \in C^k (X,\mathbb{R})$ we have that
$$\sum_{\tau \in X (k-1)}    \langle (M ')^+_{\tau, 0} (I-M^-_{\tau, 0}) \phi_\tau , \phi_\tau \rangle \leq (k+1) \mu_k \Vert \phi \Vert^2,$$
and
$$\sum_{\tau \in X (k-1)}   \left\vert \langle (M ')^+_{\tau, 0} (I-M^-_{\tau, 0}) \phi_\tau , \psi_\tau \rangle \right\vert \leq (k+1) \max \lbrace \mu_k, -\nu_k\rbrace   \Vert \phi \Vert \Vert \psi \Vert.$$
\end{lemma}

\begin{proof}
Let $\phi, \psi$ be as above. Recall that for every $\tau \in X (k-1)$, $\phi_\tau$ decomposes orthogonally as
$$\phi_\tau = (I-M^-_{\tau, 0}) \phi_\tau + M^-_{\tau, 0} \phi_\tau,$$
Therefore
\begin{dmath*}
\langle (M ')^+_{\tau, 0} (I-M^-_{\tau, 0}) \phi_\tau , \phi_\tau \rangle =
\langle (M ')^+_{\tau, 0} (I-M^-_{\tau, 0}) \phi_\tau , (I-M^-_{\tau, 0}) \phi_\tau \rangle + \langle (M ')^+_{\tau, 0} (I-M^-_{\tau, 0}) \phi_\tau , M^-_{\tau, 0} \phi_\tau \rangle.
\end{dmath*}
As explained above, $(M ')^+_{\tau, 0} (I-M^-_{\tau, 0}) \phi_\tau \in \im (I-M^-_{\tau, 0})$ and therefore
$$\langle (M ')^+_{\tau, 0} (I-M^-_{\tau, 0}) \phi_\tau , M^-_{\tau, 0} \phi_\tau \rangle = 0.$$
This yields that
\begin{dmath*}
\langle (M ')^+_{\tau, 0} (I-M^-_{\tau, 0}) \phi_\tau , \phi_\tau \rangle =
\langle (M ')^+_{\tau, 0} (I-M^-_{\tau, 0}) \phi_\tau , (I-M^-_{\tau, 0}) \phi_\tau \rangle.
\end{dmath*}
Note that by the definition of $\mu_k$
$$\langle (M ')^+_{\tau, 0} (I-M^-_{\tau, 0}) \phi_\tau , (I-M^-_{\tau, 0}) \phi_\tau \rangle \leq \mu_k \Vert (I-M^-_{\tau, 0}) \phi_\tau \Vert^2 \leq \mu_k \Vert \phi_\tau \Vert^2.$$

Summing over all $\tau \in X (k-1)$ and applying Proposition \ref{localization proposition} yields the needed results, i.e.,
\begin{dmath*}
\sum_{\tau \in X (k-1)}    \langle (M ')^+_{\tau, 0} (I-M^-_{\tau, 0}) \phi_\tau , \phi_\tau \rangle =
\sum_{\tau \in X (k-1)}    \langle (M ')^+_{\tau, 0} (I-M^-_{\tau, 0}) \phi_\tau , (I-M^-_{\tau, 0})\phi_\tau \rangle \leq
\sum_{\tau \in X (k-1)} \mu_k \Vert \phi_\tau \Vert^2 = (k+1) \mu_k \Vert \phi \Vert^2.
\end{dmath*}

The second inequality follows from Cauchy-Schwarz:
\begin{dmath*}
{\sum_{\tau \in X (k-1)}   \left\vert \langle (M ')^+_{\tau, 0} (I-M^-_{\tau, 0}) \phi_\tau , \psi_\tau \rangle \right\vert  \leq
\sum_{\tau \in X (k-1)}  \Vert (M ')^+_{\tau, 0} (I-M^-_{\tau, 0}) \phi_\tau \Vert \Vert \psi_\tau \Vert \leq } \\
\sum_{\tau \in X (k-1)}  \max \lbrace \mu_k, - \nu_k \rbrace \Vert \phi_\tau \Vert \Vert \psi_\tau \Vert \leq
\max \lbrace \mu_k, - \nu_k \rbrace \left( \sum_{\tau \in X (k-1)}   \Vert \phi_\tau  \Vert^2 \right)^{\frac{1}{2}}  \left( \sum_{\tau \in X (k-1)}   \Vert \psi_\tau \Vert^2 \right)^{\frac{1}{2}} = \\
(k+1) \max \lbrace \mu_k, - \nu_k \rbrace  \Vert \phi \Vert \Vert \psi \Vert.
\end{dmath*}
\end{proof}

We recall the following definition from the introduction:
\begin{definition}[Local spectral expander]
A $n$-dimensional complex $X$ is a {\em one-sided $\lambda$-local-spectral expander} if for every $0 \leq k \leq n-1$,  $\mu_{k} \leq \lambda$.
A $n$-dimensional complex $X$ is a {\em two-sided $\lambda$-local-spectral expander} if for every $0 \leq k \leq n-1$,  $-\lambda \leq \mu_{k} \leq \lambda$.
\end{definition}

Next, we recall the following result appearing in \cite{Opp-LocalSpectral}[Lemma 5.1] (see also \cite{Opp-weighted}[Proposition 3.7]):
\begin{lemma}
\label{spec descent lemma}
Let $X$ be a weighted pure $n$-dimensional simplicial complex, such that all the links of $X$ of dimension $\geq 1$ (including $X$ itself) are connected, then for every $0 \leq k \leq n-2$,
$$\mu_k \leq \dfrac{\mu_{k+1}}{1-\mu_{k+1}},$$
$$\nu_k \geq \dfrac{\nu_{k+1}}{1-\nu_{k+1}}.$$
\end{lemma}

A simple induction leads to the following:
\begin{corollary}
\label{spec gap descent induction coro}
Let $X$ be a weighted pure $n$-dimensional simplicial complex, such that all the links of $X$ of dimension $\geq 1$ (including $X$ itself) are connected, then for every $0 \leq k \leq n-2$,
$$\mu_k \leq \dfrac{\mu_{n-1}}{1-(n-1-k)\mu_{n-1}},$$
$$\nu_k \geq \dfrac{\nu_{n-1}}{1-(n-1-k)\nu_{n-1}}.$$
\end{corollary}

A corollary of the above corollary is the following:
\begin{corollary}
\label{Explicit spec descent coro}
Let $X$ be a weighted pure $n$-dimensional simplicial complex, such that all the links of $X$ of dimension $\geq 1$ (including $X$ itself) are connected, and $0 < \lambda \leq 1$ be some constant. If $\mu_{n-1} \leq \frac{\lambda}{1+(n-1)\lambda}$, then $X$ is a one-sided $\lambda$-spectral expander. Moreover, if $\mu_{n-1} \leq \frac{\lambda}{1+(n-1)\lambda}$ and $\frac{-\lambda}{1+(n-1)\lambda} \leq \nu_{n-1}$, then $X$ is a two-sided $\lambda$-spectral expander
\end{corollary}

\begin{proof}
By the above corollary, if $\mu_{n-1} \leq \frac{\lambda}{1+(n-1)\lambda}$ then for every $0 \leq k \leq n-2$ we have that
\begin{dmath*}
\mu_k \leq \dfrac{\mu_{n-1}}{1-(n-1-k)\mu_{n-1}} \leq \dfrac{\frac{\lambda}{1+(n-1)\lambda}}{1-(n-1-k)\frac{\lambda}{1+(n-1)\lambda}} \leq \dfrac{\frac{\lambda}{1+(n-1)\lambda}}{1-(n-1)\frac{\lambda}{1+(n-1)\lambda}} = \lambda,
\end{dmath*}
and therefore $X$ is a one-sided $\lambda$-spectral expander. The proof of the second assertion is similar.
\end{proof}

\begin{remark}
The reader should note that in \cite{Opp-LocalSpectral}[Lemma 5.1] the results of Lemma \ref{spec descent lemma} are phrased in the language of spectral gaps of the Laplacians on the links $\Delta^+_{\tau,0}$ and not as the spectral gaps of $(M ')^+_{\tau, 0}$. However, the translation of the spectral gaps is easy once one recalls that $\Delta^+_{\tau,0} = I-(M ')^+_{\tau, 0}$.
\end{remark}

\section{Decomposition theorems for upper random walks}
\label{Decomposition theorem for upper random walks sec}

Roughly speaking, we will show below that given $0 \leq k \leq n-1$, the space of $k$-cochains orthogonal to the constants can be decomposed into ``orthogonal'' parts coming from the degrees $0 \leq j \leq k$ of the simplicial complex and this decomposition allows us to bound the shrinkage of $M_k^+$. We will prove this type of results under two sets of assumptions: first, we will prove a decomposition theorem under the assumption of one-sided local spectral gap. Under this assumption the decomposition is not really an orthogonal decomposition, but it already gives a bound for the maximal non-trivial eigenvalue of $M_k^+$ and some insight for what type of cochains $M_k^+$ shrinks better than the bound given by this eigenvalue. Second, we will assume prove a decomposition theorem under the assumption of two-sided local spectral gap. Under this more restrictive assumption, we are able to give a rather comprehensive description of the spectral theory of $M_k^+$. Namely, we show that the spectrum of $M_k^+$ is concentrate in small interval centered at $\frac{j}{k+2}, j=1,...,k+1$ and that we give an explicit orthogonal decomposition that approximates the decomposition into eigenspaces.

\subsection{The space of cochains orthogonal to the constants}
For every $0 \leq k \leq n-1$, we denote $C^k_0 (X,\mathbb{R})$ to be
$$C^k_0 (X, \mathbb{R}) = \left\lbrace \phi \in C^k (X, \mathbb{R}) : \sum_{\sigma \in X (k)} m(\sigma) \phi (\sigma) = 0 \right\rbrace.$$
Let $\mathbbm{1}_k$ to be the constant $1$ function in $C^k (X,\mathbb{R})$, then by definition  for every $\phi \in C^k_0 (X,\mathbb{R})$, we have that
$$\langle \phi, \mathbbm{1}_k \rangle = \sum_{\sigma \in X (k)} m(\sigma) \phi (\sigma)=0,$$
and one can see that $C^k (X,\mathbb{R})$ has the orthogonal decomposition $C^k (X,\mathbb{R}) = \Span \lbrace \mathbbm{1}_k \rbrace \oplus C^k_0 (X,\mathbb{R})$. It is easy to see that $M^\pm_k \mathbbm{1}_k = \mathbbm{1}_k$ and since, by Corollary \ref{connection between d*d and M coro}, $M^+_k, M^-_k$ are self-adjoint operators, is follows that $M^\pm_k (C^k_0 (X,\mathbb{R})) \subseteq C^k_0 (X,\mathbb{R})$.

\begin{lemma}
\label{d, d^* preserve C_0 lemma}
For $0 \leq k \leq n-1$, $\ker ((d_{k-1})^*) \subseteq C^k_0 (X,\mathbb{R})$ and for every $\psi \in C^{k-1} (X,\mathbb{R})$, $\psi \in C^{k-1}_0 (X,\mathbb{R})$ if and only if $d_{k-1} \psi \in C^k_0 (X,\mathbb{R})$.
\end{lemma}

\begin{proof}
We note that by definition $d_{k-1} \mathbbm{1}_{k-1} = (k+1) \mathbbm{1}_{k}$, and, by Lemma \ref{d* lemma}, $(d_{k-1})^* \mathbbm{1}_k = \mathbbm{1}_{k-1}$. Therefore for every $\phi \in \ker ((d_{k-1})^*)$, we have that
$$0=\langle (d_{k-1})^* \phi, \mathbbm{1}_{k-1}  \rangle = \langle \phi, (d_{k-1}) \mathbbm{1}_{k-1}  \rangle = (k+1) \langle \phi, \mathbbm{1}_{k}  \rangle \Rightarrow \phi \in C^k_0 (X,\mathbb{R}).$$

Second, let $\psi \in C^{k-1} (X,\mathbb{R})$. If $\psi \in C^{k-1}_0 (X,\mathbb{R})$, then $M^+_{k-1} \psi \in  C^{k-1}_0 (X,\mathbb{R})$ and therefore by Corollary \ref{connection between d*d and M coro}
$$0 = \langle M^+_{k-1} \psi, \mathbbm{1}_{k-1} \rangle = \langle d_{k-1} \psi, \frac{1}{k+1} d_{k-1} \mathbbm{1}_{k-1} \rangle =  \langle d_{k-1} \psi, \mathbbm{1}_{k} \rangle,$$
i.e., $d_{k-1} \psi \in C^k_0 (X,\mathbb{R})$.

Conversely, assume that $d_{k-1} \psi \in C^k_0 (X,\mathbb{R})$, then
$$0=\langle d_{k-1} \psi, \mathbbm{1}_{k}  \rangle = \langle \psi, (d_{k-1})^* \mathbbm{1}_{k}  \rangle = \langle \psi, \mathbbm{1}_{k-1}  \rangle \Rightarrow \psi \in C^{k-1}_0 (X,\mathbb{R}).$$
\end{proof}

\subsection{Decomposition Theorem for one-sided local spectral expanders}
\begin{theorem}[Decomposition Theorem]
\label{Decomp thm}
For every $0 \leq k \leq n-1$ and every $\phi \in C^k_0 (X, \mathbb{R})$, there are $\phi^k \in C^k_0 (X, \mathbb{R}), \phi^{k-1}, (\phi^{k-1})' \in C^{k-1}_0 (X, \mathbb{R}),..., \phi^0, (\phi^0)' \in C^{0}_0 (X, \mathbb{R})$ such that if we denote $(\phi^k)' = \phi$, then the following holds:
\begin{enumerate}
\item For every $0 \leq j \leq k$,
$$\Vert (\phi^j)' \Vert^2 = \Vert \phi^j \Vert^2 + \Vert \phi^{j-1} \Vert^2 + ...+  \Vert \phi^{0} \Vert^2.$$
\item
$$\Vert d_k \phi \Vert^2 = \sum_{j=0}^k (k+1-j) \Vert \phi^j \Vert^2 + \sum_{j=0}^k \sum_{\tau \in X (j-1)}    \langle (M ')^+_{\tau, 0} (I-M^-_{\tau, 0}) (\phi^j)'_\tau , (\phi^j)'_\tau \rangle.$$
\end{enumerate}
\end{theorem}



\begin{proof}
We will prove the theorem by induction on $k$. For $k=0$ and $\phi \in C^0_0 (X,\mathbb{R})$, we take $\phi^0 = \phi$ and check that the theorem holds for this choice.
\begin{enumerate}
\item This condition holds trivially.
\item We note that $\phi \in C^0_0 (X,\mathbb{R})$ implies that $d^* \phi =0$ and therefore this condition follows from Proposition \ref{d^2 proposition}.
\end{enumerate}

Assume next that $k>0$ and that the theorem holds for $k-1$. For $\phi \in C^k_0 (X,\mathbb{R})$, we first decompose $\phi$ as $\phi = \phi^k + \phi'$, where $\phi^k \in \ker ((d_{k-1})^*)$ and $\phi' \in (\ker (d_{k-1})^*))^\perp$. This is an orthogonal decomposition and therefore
$$\Vert \phi \Vert^2 = \Vert \phi^k \Vert^2 + \Vert \phi' \Vert^2.$$
Also, by Proposition \ref{d^2 proposition},
\begin{equation}
\label{d^2 equation}
\Vert d_k \phi \Vert^2 = \Vert \phi \Vert^2 + \Vert (d_{k-1})^* \phi' \Vert^2 + \sum_{\tau \in X (k-1)}    \langle (M ')^+_{\tau, 0} (I-M^-_{\tau, 0}) \phi_\tau , \phi_\tau \rangle.
\end{equation}
We note that $(\ker (d_{k-1})^*))^\perp = \im (d_{k-1})$ and therefore, by using Lemma \ref{d, d^* preserve C_0 lemma}, there is $\psi \in C^{k-1}_0 (X,\mathbb{R})$ such that $d_{k-1} \psi = \phi'$. This yields that there is $\psi \in C^{k-1}_0 (X,\mathbb{R})$, such that $\Vert d_{k-1} \psi \Vert^2 = \Vert \phi' \Vert^2$ and
$$\Vert d_k \phi \Vert^2 = \Vert \phi \Vert^2 + \Vert (d_{k-1})^* d_{k-1} \psi \Vert^2 + \sum_{\tau \in X (k-1)}    \langle (M ')^+_{\tau, 0} (I-M^-_{\tau, 0}) \phi_\tau , \phi_\tau \rangle.$$

We recall that since $(d_{k-1})^* d_{k-1}$ is a self-adjoint operator, with non negative eigenvalues, $\sqrt{(d_{k-1})^* d_{k-1}}$ is the self-adjoint operator, with non negative eigenvalues defined as follows: for every eigenfunction $\varphi$ of $(d_{k-1})^* d_{k-1}$ with an eigenvalue $\mu$, $\varphi$ is an eigenfunction of $\sqrt{(d_{k-1})^* d_{k-1}}$ with the eigenvalue $\sqrt{\mu}$.

We will take $(\phi^{k-1})'=\sqrt{(d_{k-1})^* d_{k-1}} \psi$ and check that the theorem holds for this choice.

First, we note that, using Corollary \ref{connection between d*d and M coro}, $(d_{k-1})^* d_{k-1} (C^{k-1}_0 (X,\mathbb{R})) \subseteq C^{k-1}_0 (X,\mathbb{R})$, and therefore $\sqrt{(d_{k-1})^* d_{k-1}} (C^{k-1}_0 (X,\mathbb{R})) \subseteq C^{k-1}_0 (X,\mathbb{R})$, which implies that $(\phi^{k-1})'=\sqrt{(d_{k-1})^* d_{k-1}} \psi \in C^{k-1}_0 (X,\mathbb{R})$.

Second, we note that
\begin{dmath*}
\Vert d_{k-1} \psi \Vert^2 = \langle (d_{k-1})^* d_{k-1} \psi, \psi \rangle =
\\
 \langle \sqrt{(d_{k-1})^* d_{k-1}} \psi, \sqrt{(d_{k-1})^* d_{k-1}} \psi \rangle = \Vert\sqrt{(d_{k-1})^* d_{k-1}} \psi \Vert^2.
\end{dmath*}
Therefore, $\sqrt{(d_{k-1})^* d_{k-1}} \psi \in C^{k-1}_0 (X,\mathbb{R})$ and $\Vert \sqrt{(d_{k-1})^* d_{k-1}} \psi \Vert = \Vert \phi' \Vert$. This yields that
$$\Vert \phi \Vert^2 = \Vert \phi^k \Vert^2 + \Vert (\phi^{k-1})' \Vert^2,$$
and by the induction assumption
\begin{equation}
\label{phi as phi^j}
\Vert \phi \Vert^2 = \Vert \phi^k \Vert^2 + \Vert \phi^{k-1} \Vert^2 + ...+  \Vert \phi^{0} \Vert^2.
\end{equation}

Last, we note that
\begin{dmath*}
\Vert (d_{k-1})^* \phi' \Vert^2 = \Vert (d_{k-1})^* d_{k-1} \psi \Vert^2 = \langle (d_{k-1})^* d_{k-1} \psi, (d_{k-1})^* d_{k-1} \psi \rangle = \\
\langle (d_{k-1})^* d_{k-1} \sqrt{(d_{k-1})^* d_{k-1}} \psi,  \sqrt{(d_{k-1})^* d_{k-1}} \psi \rangle = \Vert d_{k-1} \sqrt{(d_{k-1})^* d_{k-1}} \psi \Vert^2 = \\
 \Vert d_{k-1} (\phi^{k-1})' \Vert^2.
\end{dmath*}
Combining this with \eqref{d^2 equation}, we get that
$$\Vert d \phi \Vert^2 = \Vert \phi \Vert^2 + \Vert d_{k-1} (\phi^{k-1})' \Vert^2 + \sum_{\tau \in X (k-1)}    \langle (M ')^+_{\tau, 0} (I-M^-_{\tau, 0}) \phi_\tau , \phi_\tau \rangle.$$
By the induction assumption,
$$\Vert d_{k-1} (\phi^{k-1})' \Vert^2 = \sum_{j=0}^{k-1} (k-j) \Vert \phi^j \Vert^2 + \sum_{j=0}^{k-1} \sum_{\tau \in X (j-1)}    \langle (M ')^+_{\tau, 0} (I-M^-_{\tau, 0}) (\phi^j)'_\tau , (\phi^j)'_\tau \rangle.$$
Therefore
\begin{dmath*}
\Vert d \phi \Vert^2 = \Vert \phi \Vert^2 + \Vert d_{k-1} (\phi^{k-1})' \Vert^2 + \sum_{\tau \in X (k-1)}    \langle (M ')^+_{\tau, 0} (I-M^-_{\tau, 0}) \phi_\tau , \phi_\tau \rangle = \\
 \Vert \phi \Vert^2 + \sum_{j=0}^{k-1} (k-j) \Vert \phi^j \Vert^2 + \sum_{j=0}^{k-1} \sum_{\tau \in X (j-1)}    \langle (M ')^+_{\tau, 0} (I-M^-_{\tau, 0}) (\phi^j)'_\tau , (\phi^j)'_\tau \rangle + \\
\sum_{\tau \in X (k-1)}    \langle (M ')^+_{\tau, 0} (I-M^-_{\tau, 0}) \phi_\tau , \phi_\tau \rangle  = \\
\Vert \phi \Vert^2 + \sum_{j=0}^{k-1} (k-j) \Vert \phi^j \Vert^2 + \sum_{j=0}^{k} \sum_{\tau \in X (j-1)}    \langle (M ')^+_{\tau, 0} (I-M^-_{\tau, 0}) (\phi^j)'_\tau , (\phi^j)'_\tau \rangle = \\
\sum_{j=0}^{k} (k+1-j) \Vert \phi^j \Vert^2 + \sum_{j=0}^{k} \sum_{\tau \in X (j-1)}    \langle (M ')^+_{\tau, 0} (I-M^-_{\tau, 0}) (\phi^j)'_\tau , (\phi^j)'_\tau \rangle,
\end{dmath*}
where the last equality is due to \eqref{phi as phi^j}.
\end{proof}

\begin{corollary}
\label{bound on d phi using spectral gaps}
Let $X$ be a pure $n$-dimensional weighted simplicial complex such that all the links of $X$ of dimension $\geq 1$ are connected (including $X$ itself) and let $0 \leq k \leq n-1$. Then for every $\phi \in C^k_0 (X,\mathbb{R})$, there are $\phi^k \in C^k_0 (X, \mathbb{R}), \phi^{k-1} \in C^{k-1}_0 (X, \mathbb{R}),..., \phi^0 \in C^{0}_0 (X, \mathbb{R})$, such that
$$\Vert \phi \Vert^2 = \Vert \phi^k \Vert^2 + ... + \Vert \phi^0 \Vert^2,$$
and
$$\Vert d \phi \Vert^2 \leq \sum_{j=0}^k \left( k+1-j+\sum_{i=j}^{k} (i+1)\mu_i \right) \Vert \phi^j \Vert^2.$$
In particular, if $X$ is a one-sided $\lambda$-local-spectral expander, then
$$\Vert d \phi \Vert^2 \leq \sum_{j=0}^k \left( k+1-j+ \frac{(k+j+2)(k+1-j)}{2} \lambda \right) \Vert \phi^j \Vert^2.$$
\end{corollary}

\begin{proof}
Let $\phi \in C^k_0 (X,\mathbb{R})$ and  $\phi^k \in C^k_0 (X, \mathbb{R}), \phi^{k-1}, (\phi^{k-1})' \in C^{k-1}_0 (X, \mathbb{R}),..., \phi^0, (\phi^0)' \in C^{0}_0 (X, \mathbb{R})$ as in the Decomposition Theorem. Then
$$\Vert \phi \Vert^2 = \Vert \phi^k \Vert^2 + ... + \Vert \phi^0 \Vert^2,$$
and we will prove that
$$\Vert d \phi \Vert^2 \leq \sum_{j=0}^k (k+1-j+\sum_{i=j}^{k} (i+1)\mu_i) \Vert \phi^j \Vert^2,$$
(the proof of the second inequality is similar and therefore it is left to the reader).

Note that for every $0 \leq j \leq k$, we have by Lemma \ref{bounding the sum of innerproducts lemma} that
\begin{dmath*}
\sum_{\tau \in X (j-1)} \langle (M ')^+_{\tau, 0} (I-M^-_{\tau, 0}) (\phi^j)'_\tau , (\phi^j)'_\tau \rangle \leq (j+1) \mu_j \Vert (\phi^j)' \Vert^2.
\end{dmath*}
Therefore
\begin{dmath*}
\sum_{j=0}^k \sum_{\tau \in X (j-1)} \langle (M ')^+_{\tau, 0} (I-M^-_{\tau, 0}) (\phi^j)'_\tau , (\phi^j)'_\tau \rangle \leq \sum_{j=0}^k (j+1) \mu_j \sum_{i=0}^j \Vert \phi^i \Vert^2 =   \sum_{i=0}^k \Vert \phi^i \Vert^2 \sum_{j=i}^k (j+1) \mu_j.
\end{dmath*}
Replacing the roles of $i$ and $j$ in the above inequality and combining it with the equation if the Decomposition Theorem for $\Vert d \phi \Vert^2$ yields the needed inequality.

\end{proof}

A consequence of this corollary is the following mixing results for $\lambda$ local spectral expanders:
\begin{theorem}[Mixing of the random walks]
\label{mixing of random walks thm}
Let $X$ be a weighted pure $n$-dimensional simplicial complex and let $0 \leq \lambda \leq 1$ be some constant. If $X$ is a one-sided $\lambda$-local spectral expander, then for every $0 \leq k \leq n-1$.
$$\forall \phi \in C_0^k (X,\mathbb{R}), \Vert M^+_k \phi \Vert \leq \left( \dfrac{k+1}{k+2} + \frac{k+1}{2} \lambda \right) \Vert \phi \Vert.$$

\end{theorem}

\begin{proof}

 Let $0 \leq k \leq n-1$ and $\phi \in C_0^k (X,\mathbb{R})$. Assume that $X$ is a one-sided $\lambda$-local spectral expander, then by Corollary \ref{bound on d phi using spectral gaps} we get
\begin{dmath}
\label{lazy random walk ineq}
\Vert d \phi \Vert^2 \leq \sum_{i=0}^k \left( k+1-i+ \frac{(k+i+2)(k+1-i)}{2} \lambda \right) \Vert \phi^i \Vert^2 \leq \\
\sum_{i=0}^k \left( k+1 + \frac{(k+2)(k+1)}{2} \lambda \right) \Vert \phi^i \Vert^2 = \left(k+1 + \frac{(k+2)(k+1)}{2} \lambda \right) \Vert \phi \Vert^2.
\end{dmath}
Recall that by corollary \ref{connection between d*d and M coro} $(d_{k})^* d_{k} = (k+2) M_k^+$ and therefore
$$\langle M_k^+ \phi, \phi \rangle = \dfrac{1}{k+2} \Vert d \phi \Vert^2 \leq  (\dfrac{k+1}{k+2} + \dfrac{k+1}{2} \lambda )\Vert \phi \Vert^2.$$
$M_k^+$ is a positive operator that maps $C_0^k (X,\mathbb{R})$ into itself and by the above inequality, any eigenvector of $M_k^+$ in $C_0^k (X,\mathbb{R})$ has an eigenvalue $\leq \frac{k+1}{k+2} + (k+1)\lambda$.

\end{proof}

\subsection{Decomposition Theorem for two-sided local spectral expanders}

Let $X$ a weighted simplicial complex of dimension $n$ and $0 \leq k \leq n-1$. Define the following subspaces of $C_0^k (X)$:
\begin{enumerate}
\item For $k=0$, define $V_0^0 = C_0^k (X)$.
\item For $1 \leq k \leq n-1$ and $j =0,...,k-1$, define $V_k^j = d_{k-1} ... d_{j} (C_0^j (X))$ and also define $V_k^k = (d_{k-1} (C_0^{k-1} (X)))^\perp = (V_k^{k-1})^\perp = \ker (d_{k-1}^*)$.
\end{enumerate}
By Lemma \ref{d, d^* preserve C_0 lemma} $V_k^0,...,V_k^k$ are subspaces of $C_0^k (X)$. Denote
$$U_k^j = \begin{cases}
V_k^0 & j=0 \\
V_k^{j} \cap  (V_k^{j-1})^\perp & j=1,...,k-1 \\
V_k^k & j =k
\end{cases}.$$
We note that for every $2 \leq k$ and every $0 \leq j < j+1 \leq k-1$, $V_k^{j} \subseteq V_k^{j+1}$ and therefore the following is an orthogonal decomposition:
$$C_0^k (X) = U_k^k \oplus U_k^{k-1} \oplus ... \oplus U_k^0.$$

Conceptually, recall that in the introduction we defined the incidence operator of the $j$-simplices in $k$-simplices as follows: for $0 \leq j < k \leq n$, define $d_{j \nearrow k}: C^j (X) \rightarrow C^k (X)$ to be the operator
$$\forall \sigma \in X(k), d_{j \nearrow k} \phi (\sigma) = \sum_{\tau \in X(j), \tau \subset \sigma} \phi (\tau).$$
We note that by definition $d_{k-1 \nearrow k} = d_{k-1}$ and we prove below that for every $j < k-1$, $d_{j \nearrow k} = \frac{1}{(k-j)!} d_{k-1} ... d_j$.
\begin{proposition}
For $0 \leq j < k \leq n$ and $0 \leq j \leq k-1$, $d_{j \nearrow k} = \frac{1}{(k-j)!}d_{k-1} ... d_j$
\end{proposition}

\begin{proof}
For $j=k-1$, $d_{k-1 \nearrow k} = d_{k-1}$ by definition. We proceed by downward induction. Assume that $d_{j+1 \nearrow k} = \frac{1}{(k-(j+1))!}d_{k-1} ... d_{j+1}$, then for every $\phi \in C^j (X)$ and every $\sigma \in X(k)$,
\begin{dmath*}
(d_{k-1} ... d_{j} \phi ) (\sigma) =
(d_{k-1} ... d_{j+1} (d_j \phi )) (\sigma) =
(k-1-j)! \sum_{\eta \in X(j+1), \eta \subset \sigma} (d_j \phi) (\eta) =
(k-1-j)! \sum_{\eta \in X(j+1), \eta \subset \sigma} \sum_{\tau \in X(j), \tau \subseteq \eta} \phi (\tau) =
(k-1-j)! \sum_{\tau \in X(j), \tau \subset \sigma} \phi (\tau) \left(\sum_{\eta \in X(j+1), \eta \subseteq \sigma, \tau \subseteq \eta} 1 \right) =
(k-1-j)!(k-j) \sum_{\tau \in X(j), \tau \subset \sigma} \phi (\tau) =
(k-j)! d_{j \nearrow k} \phi (\sigma),
\end{dmath*}
as needed.
\end{proof}

As a conclusion of the above Proposition,  $V_k^j = d_{j \nearrow k} (C_0^j (X))$, i.e., this definition of the spaces $U_k^j$ coincides with the one given in the introduction.

\begin{theorem}
\label{e.s decomp thm}
Let $X$ is a two-sided $\lambda$-local spectral expander, and let $\lbrace \varepsilon_{k} : 0 \leq k \leq n-1 \rbrace$ be constants defined as
$$\varepsilon_k = \begin{cases}
 \lambda & k=0 \\
2k(1+2k \sqrt{k}) \varepsilon_{k-1} + (k+1) \lambda & 0 < k
\end{cases}.$$
If $\varepsilon_k \leq \frac{1}{2(1+2 (k+1)\sqrt{k+1})}$ for all $0 \leq k \leq n-2$, then for every $0 \leq k \leq n-1$, every $0 \leq j \leq k$ and every $\phi \in C_0^k (X)$,
$$\Vert d_{k}^* d_{k} P_{U^j_k} \phi - (k+1-j) P_{U^j_k} \phi \Vert \leq \varepsilon_k \Vert P_{U^j_k} \phi \Vert.$$
\end{theorem}

\begin{remark}
Note that $\lim_{\lambda \rightarrow 0} \varepsilon_k =0$ and therefore the conditions of the Theorem holds if $\lambda$ is small enough with respect to $n$.
\end{remark}

\begin{proof}
The proof is by induction on $k$.

For $k=0$, by definition $\phi \in C_0^0 (X)$ is equal to $P_{U^0_0} \phi$. By our assumption, $\phi$ is orthogonal to the constant functions on $X(0)$ and $d^*d \phi = 2 M^{+}_0 \phi$, where $M^+_0$ is the lazy random walk on the vertices. Since $X$ is a two-sided $\lambda$-local spectral expander, it follows that  the spectrum of $2 M^{+}_0$ on $C_0^0 (X)$ is contained in the interval $[1-\lambda, 1 +  \lambda]$. Therefore
$$\Vert  d^*d \phi - \phi \Vert \leq  \lambda \Vert \phi \Vert,$$
i.e., for $k=j=0$, $\Vert  d^*d P_{U^0_0} \phi - (0+1-0) P_{U^0_0} \phi \Vert \leq \varepsilon_0 \Vert P_{U^0_0} \phi \Vert$ as needed.

Let $1 \leq k \leq n-1$ and assume that for every $\psi \in C^{k-1}_0 (X)$ and every $j=0,...,k-1$,
$$\Vert d_{k-1}^* d_{k-1} P_{U^j_{k-1}} \psi - (k-j) P_{U^j_{k-1}} \psi \Vert \leq \varepsilon_{k-1} \Vert P_{U^j_{k-1}} \psi \Vert.$$
Fix $\phi \in C_0^k (X)$ and $0 \leq j \leq k$. Denote $\varphi =P_{U^j_{k}} \phi$. With this notation, we need to prove that for every $\varphi ' \in C_0^k (X)$
$$\left\Vert \langle  d_{k}^* d_{k} \varphi - (k+1-j) \varphi, \varphi ' \rangle \right\Vert \leq \varepsilon_k \Vert \varphi \Vert \Vert \varphi ' \Vert .$$
By Proposition \ref{d^2 proposition}, this is equivalent to proving
$$\left\vert \langle d^* \varphi, d^* \varphi ' \rangle - (k-j) \langle \varphi, \varphi ' \rangle + \sum_{\tau \in X (k-1)}    \langle (M ')^+_{\tau, 0} (I-M^-_{\tau, 0}) \varphi_\tau , \varphi_\tau ' \rangle \right\vert \leq \varepsilon_k^j \Vert \varphi \Vert \Vert \varphi ' \Vert.$$
By Lemma \ref{bounding the sum of innerproducts lemma},
$$\sum_{\tau \in X (k-1)}   \left\vert \langle (M ')^+_{\tau, 0} (I-M^-_{\tau, 0}) \varphi_\tau , \varphi_\tau ' \rangle \right\vert \leq (k+1) \lambda   \Vert \varphi \Vert \Vert \varphi ' \Vert,$$
and we are left with bounding
$$\left\vert \langle d^*_{k-1} \varphi, d^*_{k-1} \varphi ' \rangle - (k-j) \langle \varphi, \varphi ' \rangle \right\vert.$$
If $j=k$, then $\varphi \in \ker (d^*_{k-1})$ and $k-j=0$ and therefore
$$\left\vert \langle d^*_{k-1} \varphi, d^*_{k-1} \varphi ' \rangle - (k-j) \langle \varphi, \varphi ' \rangle \right\vert = 0.$$
Therefore, when $j =k$,
$$\left\vert \langle d^* \varphi, d^* \varphi ' \rangle - (k-j) \langle \varphi, \varphi ' \rangle + \sum_{\tau \in X (k-1)}    \langle (M ')^+_{\tau, 0} (I-M^-_{\tau, 0}) \varphi_\tau , \varphi_\tau ' \rangle \right\vert \leq (k+1) \lambda \Vert \varphi \Vert \Vert \varphi ' \Vert,$$
as needed.

We will complete the proof by assuming that $j <k$ and showing
$$\left\vert \langle d^*_{k-1} \varphi, d^*_{k-1} \varphi ' \rangle - (k-j) \langle \varphi, \varphi ' \rangle \right\vert \leq 2k(1+2k \sqrt{k}) \varepsilon_{k-1} \Vert \varphi \Vert \Vert \varphi ' \Vert.$$
We note that $j <k$ implies that $\varphi \in V_k^j \subseteq \im (d_{k-1})$ and in particular $\varphi \perp V_k^k$. Therefore
$$\langle \varphi, \varphi ' \rangle = \langle \varphi, P_{V_k^{k-1}} \varphi ' + P_{V_k^{k}} \varphi ' \rangle = \langle \varphi, P_{V_k^{k-1}} \varphi ' \rangle.$$
We also note that $d^*_{k-1} P_{V_k^k} = 0$ and therefore
$$\langle d^*_{k-1} \varphi, d^*_{k-1} \varphi ' \rangle = \langle d^*_{k-1} \varphi, d^*_{k-1} P_{V_k^{k-1}} \varphi ' \rangle.$$
As a consequence, when bounding
$$\left\vert \langle d^*_{k-1} \varphi, d^*_{k-1} \varphi ' \rangle - (k-j) \langle \varphi, \varphi ' \rangle \right\vert$$
we can assume without loss of generality that $\varphi ' \in V_k^{k-1}$. By this assumption, there are $\psi, \psi ' \in C^{k-1}_0 (X)$ such that $d_{k-1} \psi = \varphi, d_{k-1} \psi ' = \varphi '$, .i.e.,
$$\left\vert \langle d^*_{k-1} \varphi, d^*_{k-1} \varphi ' \rangle - (k-j) \langle \varphi, \varphi ' \rangle \right\vert = \left\vert \langle d^*_{k-1} d_{k-1} \psi, d^*_{k-1}  d_{k-1} \psi ' \rangle - (k-j) \langle  d_{k-1} \psi,  d_{k-1} \psi ' \rangle \right\vert$$
We note that (as in the proof of Theorem \ref{Decomp thm})
$$\Vert \varphi \Vert^2 = \langle d_{k-1} \psi, d_{k-1} \psi \rangle = \left\langle \sqrt{d_{k-1}^* d_{k-1}} \psi, \sqrt{d_{k-1}^* d_{k-1}} \psi \right\rangle = \left\Vert \sqrt{d_{k-1}^* d_{k-1}} \psi \right\Vert^2,$$
and similarly, $\Vert \varphi ' \Vert^2 =\left\Vert \sqrt{d_{k-1}^* d_{k-1}} \psi ' \right\Vert^2$. As a consequence of the above equalities, we need to show that
\begin{equation}
\label{needed ineq}
\left\vert \langle d^*_{k-1} d_{k-1} \psi, d^*_{k-1}  d_{k-1} \psi ' \rangle - (k-j) \langle  d_{k-1} \psi,  d_{k-1} \psi ' \rangle \right\vert \leq 2k(1+2k \sqrt{k}) \varepsilon_{k-1} \left\Vert  \sqrt{d_{k-1}^* d_{k-1}} \psi \right\Vert \left\Vert \sqrt{d_{k-1}^* d_{k-1}} \psi ' \right\Vert.
\end{equation}

By $\varphi \in U_k^j$, $\varphi \in \im (d_{k-1} ... d_{j}) \cap \ker (d_{j-1}^* ... d_{k-1}^*)$ for $j>0$ and $\varphi \in \im (d_{k-1} ... d_{0})$ for $j=0$. We note that by the definition of $\psi$, this implies that $\psi \in \im (d_{k-2} ... d_{j})$, i.e., $\psi = \sum_{i=0}^j P_{U_{k-1}^i} \psi$. We also note that for $j>0$,
$d_{j-1}^* ... d_{k-1}^* d_{k-1} \psi = 0$, i.e., $d_{k-1}^* d_{k-1} \psi \in \ker (d_{j-1}^* ... d_{k-2}^*) = (V_{k-1}^{j-1})^\perp$. We will use these two facts to show that the projection of $\psi$ on the subspace $\bigcup_{i=0}^{j-1} U_{k-1}^i = V_{k-1}^{j-1}$ is small. If $j=0$, this holds vacuously. Assume that $j>0$. By the definition of $U_{k-1}^i$,
$$\forall i <j, P_{V_{k-1}^{j-1}} P_{U_{k-1}^{i}} = P_{U_{k-1}^{i}}, P_{V_{k-1}^{j-1}} P_{U_{k-1}^{j}} = 0.$$
Also, $d_{k-1}^* d_{k-1} \psi \in  (V_{k-1}^{j-1})^\perp$ implies that $P_{V_{k-1}^{j-1}} d_{k-1}^* d_{k-1} \psi =0$. Using these equalities and the induction assumption yields
\begin{dmath*}
\left\Vert P_{V_{k-1}^{j-1}} \psi \right\Vert^2 \leq \left\Vert P_{V_{k-1}^{j-1}} \left( \sum_{i=0}^{j} (k-i) P_{U_{k-1}^{i}} \psi \right)  \right\Vert^2 =
\left\Vert P_{V_{k-1}^{j-1}} \left( \sum_{i=0}^{j} (k-i) P_{U_{k-1}^{i}} \psi  \right) - P_{V_{k-1}^{j-1}} d_{k-1}^* d_{k-1} \psi  \right\Vert^2 =
\left\Vert P_{V_{k-1}^{j-1}} \left( \sum_{i=0}^{j} (k-i) P_{U_{k-1}^{i}} \psi -  d_{k-1}^* d_{k-1} P_{U_{k-1}^{i}} \psi \right)  \right\Vert^2 \leq
\left\Vert P_{V_{k-1}^{j-1}} \right\Vert^2 \left\Vert \sum_{i=0}^{j} (k-i) P_{U_{k-1}^{i}} \psi -  d_{k-1}^* d_{k-1} P_{U_{k-1}^{i}} \psi \right\Vert^2 \leq
j \sum_{i=0}^{j} \left\Vert (k-i) P_{U_{k-1}^{i}} \psi -  d_{k-1}^* d_{k-1} P_{U_{k-1}^{i}} \psi \right\Vert^2 \leq
k \varepsilon_{k-1}^2 \sum_{i=0}^{j} \left\Vert P_{U_{k-1}^{i}} \psi \right\Vert^2 = k \varepsilon_{k-1}^2 \left\Vert \psi \right\Vert^2.
\end{dmath*}
This shows that $\Vert P_{V_{k-1}^{j-1}} \psi \Vert \leq \varepsilon_{k-1} \sqrt{k} \Vert \psi \Vert$. Recall that $\psi \in V_{k-1}^{j}$ and therefore
\begin{dmath}
\label{e.s bound}
\left\Vert d_{k-1}^* d_{k-1} \psi - (k-j)\psi \right\Vert \leq \left\Vert d_{k-1}^* d_{k-1} P_{U_{k-1}^{j}} \psi - (k-j) P_{U_{k-1}^{j}} \psi \right\Vert + \left\Vert d_{k-1}^* d_{k-1} P_{V_{k-1}^{j-1}} \psi - (k-j) P_{V_{k-1}^{j-1}} \psi \right\Vert \leq
\varepsilon_{k-1} \left\Vert P_{U_{k-1}^{j}} \psi \right\Vert + \left\Vert d_{k-1}^* d_{k-1} \right\Vert \left\Vert P_{V_{k-1}^{j-1}} \psi \right\Vert + (k-j) \left\Vert P_{V_{k-1}^{j-1}} \psi \right\Vert \leq \varepsilon_{k-1} (1+ k \sqrt{k} + (k-j) \sqrt{k}) \Vert \psi \Vert \leq \varepsilon_{k-1} (1+ 2 k \sqrt{k}) \Vert \psi \Vert.
\end{dmath}
This implies
\begin{dmath*}
\left\vert \langle d^*_{k-1} d_{k-1} \psi, d^*_{k-1}  d_{k-1} \psi ' \rangle - (k-j) \langle  d_{k-1} \psi,  d_{k-1} \psi ' \rangle \right\vert =
\left\vert \langle d^*_{k-1} d_{k-1} \psi, d^*_{k-1}  d_{k-1} \psi ' \rangle - (k-j) \langle   \psi, d_{k-1}^* d_{k-1} \psi ' \rangle \right\vert =
\left\vert \langle d^*_{k-1} d_{k-1} \psi - (k-j) \psi, d^*_{k-1}  d_{k-1} \psi ' \rangle  \right\vert \leq
\left\Vert d^*_{k-1} d_{k-1} \psi - (k-j) \psi \right\Vert \left\Vert d^*_{k-1}  d_{k-1} \psi ' \right\Vert \leq
\varepsilon_{k-1} (1+ 2 k \sqrt{k}) \Vert \psi \Vert \left\Vert \sqrt{d^*_{k-1}  d_{k-1}} \right\Vert \left\Vert \sqrt{d^*_{k-1}  d_{k-1}}\psi ' \right\Vert \leq
\varepsilon_{k-1} (1+ 2 k \sqrt{k}) \sqrt{k} \Vert \psi \Vert \left\Vert \sqrt{d^*_{k-1}  d_{k-1}}\psi ' \right\Vert.
\end{dmath*}
In order complete the proof of \eqref{needed ineq}, we need to show that $\Vert \psi \Vert \leq 2\sqrt{k} \Vert \sqrt{d^*_{k-1}  d_{k-1}}\psi \Vert$. Recall that by the assumptions of the Theorem $\varepsilon_{k-1} \leq \frac{1}{2(1+2 k\sqrt{k})}$ and therefore, using inequality \eqref{e.s bound},
\begin{dmath*}
2\sqrt{k} \left\Vert \sqrt{d^*_{k-1}  d_{k-1}}\psi \right\Vert \geq 2 \left\Vert \sqrt{d^*_{k-1}  d_{k-1}} \right\Vert \left\Vert \sqrt{d^*_{k-1}  d_{k-1}} \psi \right\Vert \geq
2 \left\Vert d^*_{k-1}  d_{k-1} \psi \right\Vert \geq 2 \left((k-j) \Vert  \psi \Vert  - \left\Vert  d^*_{k-1}  d_{k-1} \psi -(k-j) \psi \right\Vert\right) \geq
\left(2(k-j) -\varepsilon_{k-1} 2(1+ 2 k \sqrt{k}) \right)  \Vert  \psi \Vert \geq \Vert  \psi \Vert,
\end{dmath*}
as needed.
\end{proof}

\begin{corollary}
\label{almost e.s. corollary}
Let $X$ is a two-sided $\lambda$-local spectral expander, and let $\lbrace \varepsilon_{k} : 0 \leq k \leq n-1 \rbrace$ be constants defined in Theorem \ref{e.s decomp thm}.  If $\varepsilon_k \leq \frac{1}{2(1+2 (k+1)\sqrt{k+1})}$ for all $0 \leq k \leq n-2$, then for every $0 \leq k \leq n-1$, every $0 \leq j \leq k$ and every $\phi \in C_0^k (X)$,
$$\left\Vert M_k^+ P_{U^j_k} \phi - (\frac{k+1-j}{k+2}) P_{U^j_k} \phi \right\Vert \leq \frac{\varepsilon_k}{k+2} \Vert P_{U^j_k} \phi \Vert.$$
\end{corollary}

\begin{proof}
By Corollary \ref{connection between d*d and M coro}, $\frac{1}{k+2} d_{k}^* d_k = M_k^+$ and the inequality stated above follows.
\end{proof}

The above Corollary allows us to determine the spectrum of $M_k^+$ given that the constants $\varepsilon_k, k=0,...,n-1$ are small enough:

\begin{theorem}
\label{e.v. bounds thm}
Let $X$ is a two-sided $\lambda$-local spectral expander, and let $\lbrace \varepsilon_{k} : 0 \leq k \leq n-1 \rbrace$ be constants defined in Theorem \ref{e.s decomp thm}.  If $\varepsilon_k \leq \frac{1}{2(1+2 (k+1)\sqrt{k+1})}$ for all $0 \leq k \leq n-2$ and $\varepsilon_{n-1} < \frac{1}{2\sqrt{n}}$, then
$$\Spec (M_k^+) \subseteq \lbrace 1 \rbrace \cup \bigcup_{j=0}^{k} \left[\frac{k+1-j}{k+2} -   \frac{\sqrt{k+1}}{k+2}\varepsilon_k, \frac{k+1-j}{k+2} + \frac{\sqrt{k+1}}{k+2}\varepsilon_k \right].$$
Moreover, for $\phi \in C_0^k (X)$ such that $M_k^+ \phi = \mu \phi$, if $\mu \in \left[\frac{k+1-j}{k+2} -   \frac{\sqrt{k+1}}{k+2}\varepsilon_k, \frac{k+1-j}{k+2} + \frac{\sqrt{k+1}}{k+2}\varepsilon_k \right]$, then
$$\Vert \phi - P_{U_k^j} \phi \Vert \leq \frac{\sqrt{k+1} \varepsilon_k}{1-\sqrt{k+1} \varepsilon_k} \Vert \phi \Vert.$$
\end{theorem}

\begin{proof}
Let $\phi \in C_0^k (X)$ be an eigenvector of $M_k^+$ with eigenvalue $\mu$. Then by Corollary \ref{almost e.s. corollary} and by the fact that $U_k^0 \oplus ... \oplus U_k^k$ is an orthogonal decomposition of $C_0^k (X)$, it holds that
\begin{dmath}
\label{long comp}
\sum_{j=0}^k \left\vert \mu - \frac{k+1-j}{k+2} \right\vert^2 \left\Vert P_{U_k^j} \phi \right\Vert^2 =
\left\Vert \sum_{j=0}^k \left( \mu - \frac{k+1-j}{k+2} \right)  P_{U_k^j} \phi \right\Vert^2 =
\left\Vert \mu \phi - \sum_{j=0}^k \frac{k+1-j}{k+2}   P_{U_k^j} \phi \right\Vert^2 =
\left\Vert M_k^+ \phi - \sum_{j=0}^k \frac{k+1-j}{k+2}   P_{U_k^j} \phi \right\Vert^2 =
\left\Vert  \sum_{j=0}^k M_k^+  P_{U_k^j} \phi - \frac{k+1-j}{k+2}   P_{U_k^j} \phi \right\Vert^2 \leq
(k+1) \sum_{j=0}^k \left\Vert M_k^+  P_{U_k^j} \phi - \frac{k+1-j}{k+2}   P_{U_k^j} \phi \right\Vert^2 \leq
(k+1) \sum_{j=0}^k  \frac{\varepsilon_k^2}{(k+2)^2} \left\Vert  P_{U_k^j} \phi \right\Vert^2 =
(k+1)  \frac{\varepsilon_k^2}{(k+2)^2} \left\Vert \phi \right\Vert^2.
\end{dmath}
Therefore
$$\sum_{j=0}^k \frac{\left\Vert P_{U_k^j} \phi \right\Vert^2}{\Vert \phi \Vert^2} \left\vert \mu - \frac{k+1-j}{k+2} \right\vert^2  \leq (k+1)  \frac{\varepsilon_k^2}{(k+2)^2}.$$
Noting that
$$\sum_{j=0}^k \frac{\left\Vert P_{U_k^j} \phi \right\Vert^2}{\Vert \phi \Vert^2} =1,$$
it follows that there is $j_0$ such that $\left\vert \mu - \frac{k+1-j_0}{k+2} \right\vert  \leq \sqrt{k+1}  \frac{\varepsilon_k}{k+2}$. We also note that
\begin{dmath*}
 (k+1)  \frac{\varepsilon_k^2}{(k+2)^2} \Vert \phi \Vert^2 \geq
 \sum_{j=0, j \neq j_0}^k \left\Vert P_{U_k^j} \phi \right\Vert^2 \left\vert \mu - \frac{k+1-j}{k+2} \right\vert^2 \geq \\
 \sum_{j=0, j \neq j_0}^k \left\Vert P_{U_k^j} \phi \right\Vert^2 \left(\left\vert \frac{k+1-j}{k+2} - \frac{k+1-j_0}{k+2} \right\vert -\left\vert \mu - \frac{k+1-j_0}{k+2} \right\vert \right)^2 \geq  \\
  \sum_{j=0, j \neq j_0}^k \left\Vert P_{U_k^j} \phi \right\Vert^2 \left(\frac{1}{k+2} -\sqrt{k+1}  \frac{\varepsilon_k}{k+2}  \right)^2 =
  \sum_{j=0, j \neq j_0}^k \left\Vert P_{U_k^j} \phi \right\Vert^2 \left(\frac{1-\sqrt{k+1} \varepsilon_k}{k+2} \right)^2 = \\
 \left(\frac{1-\sqrt{k+1} \varepsilon_k}{k+2} \right)^2 \left\Vert (I-P_{U_k^{j_0}}) \phi \right\Vert^2.
\end{dmath*}
Therefore
$$\left\Vert \phi -P_{U_k^{j_0}} \phi \right\Vert \leq \frac{\sqrt{k+1} \varepsilon_k}{1-\sqrt{k+1} \varepsilon_k} \Vert \phi \Vert.$$
\end{proof}

Relying on the above Theorem, we denote $W_k^j$ to be the subspaces of $C_0^k (X)$ spanned by eigenvectors of $M_k^+$ with eigenvalues in the interval $\left[\frac{k+1-j}{k+2} -   \frac{\sqrt{k+1}}{k+2}\varepsilon_k, \frac{k+1-j}{k+2} + \frac{\sqrt{k+1}}{k+2}\varepsilon_k \right]$, i.e.,
$$W_k^j = \Spanv \left\lbrace \varphi : M_k^+ \varphi = \mu \varphi, \mu \in \left[\frac{k+1-j}{k+2} -   \frac{\sqrt{k+1}}{k+2}\varepsilon_k, \frac{k+1-j}{k+2} + \frac{\sqrt{k+1}}{k+2}\varepsilon_k \right] \right\rbrace.$$
We note that if the constants $\varepsilon_k$ are small enough, then these subspaces intersect trivially and $W_k^0 \oplus ... \oplus W_k^k$ is an orthogonal decomposition of $C_0^k (X)$. Next, we show that for every $j$, and every $\phi \in C_0^k (X)$, the norm of projection of $\phi$ on $W_k^j$ can be approximate by the projection of $\phi$ on $U_k^j$.

\begin{theorem}
\label{approx of W_k by U_k thm}
Let $X$ is a two-sided $\lambda$-local spectral expander, and let $\lbrace \varepsilon_{k} : 0 \leq k \leq n-1 \rbrace$ be the constants defined in Theorem \ref{e.s decomp thm}. For every $0 \leq k \leq n-1$ and $0 \leq j_0 \leq k$, if $\varphi \in  U_k^{j_0}$, then
$$\Vert \varphi - P_{W_k^{j_0}} \varphi  \Vert \leq \sqrt{k+2} \varepsilon_k \Vert \varphi \Vert.$$
Moreover, for every $\phi \in C_0^k (X)$,
$$\left\Vert P_{W_k^j} \phi \right\Vert \leq \left\Vert P_{U_k^j} \phi \right\Vert+ \left(\sqrt{k+2} \varepsilon_k + \frac{\sqrt{k+1} \varepsilon_{k}}{1-\sqrt{k+1} \varepsilon_{k}} \right) \Vert \phi \Vert^2,$$
and
$$\left\Vert P_{W_k^j} \phi \right\Vert \geq \left\Vert P_{U_k^j} \phi \right\Vert - \left(\sqrt{k+2} \varepsilon_k + \frac{\sqrt{k+1} \varepsilon_{k}}{1-\sqrt{k+1} \varepsilon_{k}} \right) \Vert \phi \Vert^2.$$
\end{theorem}

\begin{proof}
Note that by our assumptions, $W_k^0,...,W_k^k$ have trivial intersection and since $M_{k}^+$ is self-adjoint, these are orthogonal spaces, i.e., $W_k^0 \oplus ... \oplus W_k^k$ is an orthogonal decomposition of $C_0^k (X)$. Fix some $0 \leq j_0 \leq k$ and assume that $\varphi \in U_{k}^{j_0}$. Then as in \eqref{long comp} above,
\begin{dmath*}
\frac{1}{(k+2)^2} \left\Vert \varphi - P_{W_k^{j_0}} \varphi \right\Vert^2 \leq  \sum_{j=0}^k \left\vert \frac{k+1-j_0}{k+2}- \frac{k+1-j}{k+2} \right\vert^2 \left\Vert P_{W_k^j} \varphi \right\Vert^2 =
\left\Vert  \frac{k+1-j_0}{k+2} \varphi - \sum_{j=0}^k \frac{k+1-j}{k+2}   P_{W_k^j} \varphi \right\Vert^2 =
\left\Vert \frac{k+1-j_0}{k+2} \varphi - M_k^+ \varphi + \sum_{j=0}^k M_k^+ P_{W_k^j}\varphi - \frac{k+1-j}{k+2}   P_{W_k^j} \varphi \right\Vert^2 \leq
(k+2) \left(\left\Vert \frac{k+1-j_0}{k+2} \varphi - M_k^+ \varphi \right\Vert^2 + \sum_{j=0}^k \left\Vert M_k^+ P_{W_k^j} \varphi - \frac{k+1-j}{k+2}   P_{W_k^j} \varphi \right\Vert^2 \right) \leq
(k+2) \left(\frac{\varepsilon_k^2}{(k+2)^2} \Vert \varphi \Vert^2 + \sum_{j=0}^k \frac{k+1}{(k+2)^2}\varepsilon_k^2  \Vert P_{W_k^j} \varphi \Vert^2 \right) = \frac{\varepsilon_k^2}{k+2} \Vert \varphi \Vert^2
\end{dmath*}
and therefore
\begin{equation}
\label{ineq1}
\Vert \varphi - P_{W_k^{j_0}} \varphi  \Vert \leq \sqrt{k+2} \varepsilon_k \Vert \varphi \Vert, \forall \varphi \in  U_k^{j_0}
\end{equation}
as needed.

Fix $0 \leq j \leq k$. To avoid cumbersome notation, we denote $U = U_k^j,  W = W_k^j$. Note that by \eqref{ineq1}, it holds that 
$$\Vert (I- P_W) P_U \Vert \leq \sqrt{k+2} \varepsilon_k,$$
and thus
\begin{equation}
\label{ineq2}
\Vert  P_U (I- P_W)\Vert \leq \sqrt{k+2} \varepsilon_k.
\end{equation}
By Theorem \ref{e.v. bounds thm}, it holds that 
\begin{equation*}
\Vert (I-P_{U}) P_{W} \phi \Vert \leq \frac{\sqrt{k+1} \varepsilon_{k}}{1-\sqrt{k+1} \varepsilon_{k}} \Vert P_{W} \phi \Vert,
\end{equation*}
and thus 
\begin{equation}
\label{ineq3}
\Vert  P_{W} (I-P_{U}) \Vert = \Vert (I-P_{U}) P_{W} \Vert \leq  \frac{\sqrt{k+1} \varepsilon_{k}}{1-\sqrt{k+1} \varepsilon_{k}}.
\end{equation}


We note that
\begin{dmath*}
\Vert P_{W} \phi \Vert^2 =  \langle P_{W} \phi, \phi \rangle = \\
\langle P_{W} \phi, P_{U} \phi \rangle +  \langle P_{W} \phi, (I-P_{U}) \phi \rangle = \\
\Vert P_{U} \phi \Vert^2 +   \langle (P_{W} - P_{U}) \phi, P_U \phi \rangle +  \langle (I-P_{U}) P_{W} \phi, \phi \rangle = \\
\Vert P_{U} \phi \Vert^2 +   \langle  P_{U} (P_{W} - I) \phi, \phi \rangle + \langle (I-P_{U}) P_{W} \phi, \phi \rangle.
\end{dmath*}

It follows that 
\begin{dmath*}
\Vert P_{W} \phi \Vert^2 \leq \Vert P_{U} \phi \Vert^2 +  \left\vert \langle  P_{U} (I-P_W) \phi, \phi \rangle \right\vert + \left\vert \langle (I-P_{U}) P_{W} \phi, \phi \rangle \right\vert \leq \\
\Vert P_{U} \phi \Vert^2 + \left( \Vert P_{U} (I-P_W) \Vert +  \Vert (I-P_{U}) P_{W}) \Vert \right) \Vert \phi \Vert^2 \leq^{\eqref{ineq2}, \eqref{ineq3}} \\
\Vert P_{U} \phi \Vert^2 + \left(\sqrt{k+2} \varepsilon_k + \frac{\sqrt{k+1} \varepsilon_{k}}{1-\sqrt{k+1} \varepsilon_{k}} \right) \Vert \phi \Vert^2,
\end{dmath*}
as needed. Similarly,
\begin{dmath*}
\Vert P_{W} \phi \Vert^2 \geq \Vert P_{U} \phi \Vert^2 - \left(\sqrt{k+2} \varepsilon_k + \frac{\sqrt{k+1} \varepsilon_{k}}{1-\sqrt{k+1} \varepsilon_{k}} \right) \Vert \phi \Vert^2.
\end{dmath*}
\end{proof}

The above Theorem allows us to determine the rate of decay for the iterated random walk for a cochain $\phi$ based on the size of its projection of the spaces $U_k^0,...,U_k^k$:

\begin{corollary}
\label{rw shrink coro}
Let $X$ is a two-sided $\lambda$-local spectral expander, and let $\lbrace \varepsilon_{k} : 0 \leq k \leq n-1 \rbrace$ be the constants defined above. Assume that $\varepsilon_k \leq \frac{1}{2(1+2 (k+1)\sqrt{k+1})}$ for all $0 \leq k \leq n-2$ and $\varepsilon_{n-1} < \frac{1}{2\sqrt{n}}$. For $0 \leq k \leq n-1$ and $i \in \mathbb{N}$, define
$$b_k (i , \lambda) = \sum_{j=0}^k \left(  \frac{k+1-j}{k+2} + \frac{\sqrt{k+1}}{k+2}\varepsilon_k \right)^{2i} \left( \frac{\sqrt{k+1} \varepsilon_{k}}{1-\sqrt{k+1} \varepsilon_{k}} \right).$$
Then for every $0 \leq k \leq n-1$, every $i \in \mathbb{N}$ and every $\phi \in C_0^k (X)$,
$$\Vert (M_k^+)^i \phi \Vert \leq \sqrt{\sum_{j=0}^k \left( \left(  \frac{k+1-j}{k+2} + \frac{\sqrt{k+1}}{k+2}\varepsilon_k \right)^{2i} +  b_k (i, \lambda) \right) \left\Vert P_{U_k^j} \phi \right\Vert^2},$$
and
$$\Vert (M_k^+)^i \phi \Vert \geq \sqrt{\sum_{j=0}^k \left(  \frac{k+1-j}{k+2} - \frac{\sqrt{k+1}}{k+2}\varepsilon_k \right)^{2i} (1 - b_k (i, \lambda)) \left\Vert P_{U_k^j} \phi \right\Vert^2}.$$
\end{corollary}

\begin{proof}
The proofs of the two inequalities are similar and we will prove only the first one and leave the second one to the reader. Fix $0 \leq k \leq n-1$. Every $\phi \in C_0^k$ has two orthogonal decompositions:
$$\phi = \sum_{j=0}^k P_{W_k^j} \phi ,$$
and
$$\phi = \sum_{j=0}^k P_{U_k^j} \phi .$$
By definition, $W_k^j$ is an invariant subspace spanned by eigenvectors with eigenvalues in $[ \frac{k+1-j}{k+2} - \frac{\sqrt{k+1}}{k+2}\varepsilon_k,  \frac{k+1-j}{k+2} + \frac{\sqrt{k+1}}{k+2}\varepsilon_k]$. Thus $M_k^+ P_{W_k^j} = P_{W_k^j} M_k^+ P_{W_k^j}$ and $\Vert M_k^+ P_{W_k^j} \Vert \leq \frac{k+1-j}{k+2} + \frac{\sqrt{k+1}}{k+2}\varepsilon_k$. Combining these facts yields
\begin{dmath*}
\Vert (M_k^+)^i \phi \Vert^2 = \sum_{j=0}^k  \left\Vert P_{W_k^j} (M_k^+)^i \phi \right\Vert^2 = \sum_{j=0}^k  \left\Vert (P_{W_k^j} M_k^+)^i P_{W_k^j} \phi \right\Vert^2 \leq \sum_{j=0}^k \left(\frac{k+1-j}{k+2} + \frac{\sqrt{k+1}}{k+2}\varepsilon_k \right)^{2i}  \left\Vert P_{W_k^j} \phi \right\Vert^2.
\end{dmath*}
By Theorem \ref{approx of W_k by U_k thm},
$$ \left\Vert P_{W_k^j} \phi \right\Vert^2 \leq \left\Vert P_{U_k^j} \phi \right\Vert^2 + \left(\sqrt{k+2} \varepsilon_k + \frac{\sqrt{k+1} \varepsilon_{k}}{1-\sqrt{k+1} \varepsilon_{k}} \right)  \Vert \phi \Vert^2,$$
and it follows that
\begin{dmath*}
\Vert (M_k^+)^i \phi \Vert^2 \leq \sum_{j=0}^k \left(  \frac{k+1-j}{k+2} + \frac{\sqrt{k+1}}{k+2}\varepsilon_k \right)^{2i} \left( \left\Vert P_{U_k^j} \phi \right\Vert^2 + \left( \frac{\sqrt{k+1} \varepsilon_{k}}{1-\sqrt{k+1} \varepsilon_{k}} \right)  \Vert \phi \Vert^2 \right) = \\
\sum_{j=0}^k \left(  \frac{k+1-j}{k+2} + \frac{\sqrt{k+1}}{k+2}\varepsilon_k \right)^{2i} \left\Vert P_{U_k^j} \phi \right\Vert^2 + \sum_{j=0}^k \left(  \frac{k+1-j}{k+2} + \frac{\sqrt{k+1}}{k+2}\varepsilon_k \right)^{2i} \left( \frac{\sqrt{k+1} \varepsilon_{k}}{1-\sqrt{k+1} \varepsilon_{k}} \right)  \Vert \phi \Vert^2 = \\
\sum_{j=0}^k \left(  \frac{k+1-j}{k+2} + \frac{\sqrt{k+1}}{k+2}\varepsilon_k \right)^{2i} \left\Vert P_{U_k^j} \phi \right\Vert^2 + b_k (i, \lambda)  \Vert \phi \Vert^2.
\end{dmath*}
Using the fact that 
$$\Vert \phi \Vert^2 = \sum_{j=0}^k \left\Vert P_{U_k^j} \phi \right\Vert^2,$$
completes the proof.
\end{proof}

\begin{remark}[Mistakes in the original proof and false inequalities]
The published version of this paper \cite{RWpaper} had a very silly mistake (an erratum was later submitted).  We showed that for every $j' \neq j$ it followed that 
$\Vert P_{W_k^{j'}} P_{U_k^j} \Vert \leq \sqrt{k+2} \varepsilon_k$ (which is correct), but \textbf{falsely} deduced from it that 
$\Vert P_{W_k^{j}} P_{U_k^{j'}} \phi \Vert \leq  \sqrt{k+2} \varepsilon_k \Vert P_{U_k^j} \phi \Vert$. The mistake came from a silly indexation mistake - we confused $P_{W_k^{j}} P_{U_k^{j'}}$ with $P_{W_k^{j'}} P_{U_k^j}$.   
We then got the following \textbf{false} inequalities: for every $\phi$
$$\Vert P_{W_k^{j}} \phi \Vert \leq (1+ (k+1) \sqrt{k+2} \varepsilon_k) \Vert P_{U_k^j} \phi \Vert,$$
and
$$\Vert P_{W_k^{j}} \phi \Vert \geq (1- (k+1) \sqrt{k+2} \varepsilon_k) \Vert P_{U_k^j} \phi \Vert.$$
We did not notice it when submitting the paper, but if these inequalities were correct they would imply that $U_k^j = W_k^j$. Indeed, let $\phi \perp U_k^j$, then by the first inequality it follows that $\phi \perp W_k^j$. Similarly, if $\phi \perp W_k^j$, then by the second inequality (for $\varepsilon_k$ sufficiently small) it follows that $\phi \perp U_k^j$. Thus, $(W_k^j)^{\perp} = (U_k^j)^{\perp}$ and $U_k^j = W_k^j$. However, we do not expect $U_k^j = W_k^j$ to hold in general, since $U_k^j$ are only approximations of the eigenspaces and not the actual eigenspaces.
\end{remark}

\bibliographystyle{alpha}
\bibliography{bibl}

\appendix
\section{Proof of Proposition \ref{localization proposition}}

\begin{proof}
Let $\phi, \psi \in C^l (X,\mathbb{R})$, then
\begin{dmath*}
\sum_{\tau \in X (k)} \langle \phi_\tau, \psi_\tau \rangle =
\sum_{\tau \in X (k)} \sum_{\eta \in X_\tau^{(l-k-1)}} m_\tau (\eta) \phi_\tau (\eta)\psi_\tau (\eta) = \\
\sum_{\tau \in X (k)} \sum_{\eta \in X_\tau^{(l-k-1)}} m (\tau \cup \eta) \phi (\tau \cup \eta)\psi (\tau \cup \eta) =
\sum_{\tau \in X (k)} \sum_{\sigma \in X (l), \tau \subset \sigma} m (\sigma) \phi (\sigma) \psi (\sigma) = \\
\sum_{\sigma \in X (l)} \sum_{\tau \in X (k), \tau \subset \sigma} m (\sigma) \phi (\sigma)\psi (\sigma) =
{l+1 \choose k+1} \sum_{\sigma \in X (l)} m (\sigma) \phi (\sigma) \psi (\sigma) =
{l+1 \choose k+1} \Vert \phi \Vert^2.
\end{dmath*}
In order to prove the second equality, we notice that for every $\tau \in X (k)$ and every $\eta \in X_\tau^{(l-k-2)}$, we have that
\begin{dmath*}
(d^* \phi )_\tau (\eta) = d^* \phi (\tau \cup \eta) = \sum_{\sigma \in X (l), \tau \cup \eta \subset \sigma} \dfrac{m(\sigma)}{m(\tau \cup \eta)} \phi (\sigma) =
\sum_{\sigma \setminus \tau \in X (l-k-1), \eta \subset \sigma \setminus \tau} \dfrac{m_\tau (\sigma \setminus \tau)}{m_\tau (\eta)} \phi_\tau (\sigma \setminus \tau) = d_\tau^* \phi_\tau (\eta).
\end{dmath*}
Therefore, $(d^* \phi )_\tau = d_\tau^* \phi_\tau$ and similarly, $(d^* \psi )_\tau = d_\tau^* \psi_\tau$. By the equality proven above
$${l \choose k+1} \langle d^* \phi, d^* \psi \rangle  = \sum_{\tau \in X (k)} \langle (d^* \phi )_\tau, (d^* \psi )_\tau  \rangle = \sum_{\tau \in X (k)} \langle d^*_\tau \phi_\tau,d^*_\tau \psi_\tau  \rangle.$$
Assume now that $l <n$, then for every $\sigma \in X (l+1)$, the following holds:
\begin{dmath*}
(d \phi (\sigma))(d \psi (\sigma)) =
(\sum_{\eta \in X (l), \eta \subset \sigma} \phi (\eta))(\sum_{\eta \in X (l), \eta \subset \sigma} \psi (\eta)) =
\sum_{\eta \in X (l), \eta \subset \sigma} \phi (\eta)\psi (\eta) +  \sum_{\eta, \eta' \in X (l), \eta \neq \eta', \eta, \eta' \subset \sigma} \left( \phi (\eta) \psi (\eta') + \phi (\eta ') \psi (\eta) \right) =
\sum_{\eta, \eta' \in X (l), \eta \neq \eta', \eta, \eta' \subset \sigma} (\phi (\eta) + \phi (\eta'))(\psi (\eta) + \psi (\eta'))  - l \sum_{\eta \in X (l), \eta \subset \sigma} \phi (\eta)\psi (\eta) =
\sum_{\tau \in X (l-1), \tau \subset \sigma} (d_\tau \phi_\tau (\sigma \setminus \tau))(d_\tau \psi_\tau (\sigma \setminus \tau))  - l \sum_{\eta \in X (l), \eta \subset \sigma} \phi (\eta)\psi (\eta).
\end{dmath*}
Therefore
\begin{dmath*}
\langle d \phi , d \psi \rangle =
\sum_{\sigma \in X (l+1)} m (\sigma) (d \phi (\sigma))(d \psi (\sigma)) =
\sum_{\sigma \in X (l+1)} m (\sigma) \sum_{\tau \in X (l-1), \tau \subset \sigma} (d_\tau \phi_\tau (\sigma \setminus \tau))(d_\tau \psi_\tau (\sigma \setminus \tau)) \\
- l \sum_{\sigma \in X (l+1)} m (\sigma) \sum_{\eta \in X (l), \eta \subset \sigma} \phi (\eta)\psi (\eta).
\end{dmath*}
We note that
\begin{dmath*}
\sum_{\sigma \in X (l+1)} m (\sigma) \sum_{\tau \in X (l-1), \tau \subset \sigma} (d_\tau \phi_\tau (\sigma \setminus \tau))(d_\tau \psi_\tau (\sigma \setminus \tau)) =
\sum_{\tau \in X (l-1)} \sum_{\sigma \in X (l+1), \tau \subset \sigma} m (\sigma) (d_\tau \phi_\tau (\sigma \setminus \tau))(d_\tau \psi_\tau (\sigma \setminus \tau)) =
\sum_{\tau \in X (l-1)} \sum_{\gamma \in X_\tau (1)} m_\tau (\gamma) (d_\tau \phi_\tau (\gamma))(d_\tau \psi_\tau (\gamma)) =
\sum_{\tau \in X (l-1)} \langle d_\tau \phi_\tau, d_\tau \psi_\tau \rangle,
\end{dmath*}
and also
\begin{dmath*}
l \sum_{\sigma \in X (l+1)} m (\sigma) \sum_{\eta \in X (l), \eta \subset \sigma} \phi (\eta)\psi (\eta) =
l \sum_{\eta \in X (l)} \phi (\eta)\psi (\eta)  \sum_{\sigma \in X (l+1), \eta \subset \sigma} m (\sigma) =
l \sum_{\eta \in X (l)} m(\eta) \phi (\eta)\psi (\eta) =
l \langle \phi, \psi \rangle =
 \sum_{\tau \in X (l-1)} \dfrac{l}{l+1} \langle \phi_\tau, \psi_\tau \rangle,
\end{dmath*}
where the last equality is due to the equality
$$(l+1) \langle  \phi ,\psi \rangle= \sum_{\tau \in X (l-1)} \langle \phi_\tau, \psi_\tau \rangle,$$
proven above.
\end{proof}

\end{document}